\documentclass[11pt]{amsart}
\usepackage{amsfonts,amssymb,verbatim,amsmath,amsthm,latexsym,textcomp,amscd}
\usepackage{latexsym,amsfonts,amssymb,epsfig,verbatim}
\usepackage{amsmath,amsthm,amssymb,latexsym,graphics,textcomp,hyperref,enumerate}
\usepackage[capitalize]{cleveref}
\usepackage{hyperref}
\usepackage{comment}
\usepackage[utf8]{inputenc}
\usepackage{xcolor}
\usepackage{eucal}
\title{Thurston's Bounded image theorem}
\author{Cyril Lecuire and Ken'ichi Ohshika}
\address{Laboratoire Emile Picard,
Universit\'{e} Paul Sabatier,
118 Route de Narbonne,
31062 Toulouse Cedex 4, France, and
Department of Mathematics, Faculty of Science, Gakushuin University, Toshima-ku, Tokyo 171-8588, Japan}
\newcommand{\complexes}{\mathbb{C}}
\newcommand{\reals}{\mathbb{R}}

\newcommand{\integers}{\mathbb{Z}}
\newcommand{\naturals}{\mathbb{N}}

\newcommand{\hyperbolic}{\mathbb{H}}
\newcommand{\hthree}{\hyperbolic^3}

\newcommand{\PSL}{{\rm PSL}_2(\mathbb{C})}

\newcommand{\ie}{i.e.\ }

\newcommand{\Fr}{\operatorname{Fr}}
\newcommand{\Int}{\operatorname{Int}}
\newcommand{\len}{\mathrm{length}}
\newcommand{\cc}{\mathcal{CC}}
\newcommand{\base}{\mathbf{base}}

\newcommand{\AH}{\mathsf{AH}}
\newcommand{\QF}{\mathsf{QF}}
\newcommand{\QH}{\mathsf{QH}}

\newcommand{\teich}{\mathcal{T}}
\let \pslc \PSL

\newcommand{\eps}{\epsilon}

\newtheorem{theorem}{Theorem}[section]
\newtheorem{proposition}[theorem]{Proposition}
\newtheorem{lemma}[theorem]{Lemma}
\newtheorem{corollary}[theorem]{Corollary}
\newtheorem{claim}[theorem]{Claim}

\theoremstyle{definition}
\newtheorem{definition}[theorem]{Definition}

\newtheorem{setting}[theorem]{Setting}

\includecomment{comments}
\specialcomment{notes}{\color{red}}{\color{black}}
\excludecomment{generalization}
\begin{document}
\sloppy
\maketitle
\begin{abstract}
Thurston's bounded image theorem is one of the key steps in his proof of the uniformisation theorem  for Haken manifolds.
Thurston never published its proof, and no proof has been known up to today, although a proof of  its weaker version, called the bounded orbit theorem is known.
In this paper, we give a proof of the original bounded image theorem, relying on recent development of Kleinian group theory.
\end{abstract}

\tableofcontents

\section{Introduction}
\label{intro}
From the late 1970s to the early 1980s, Thurston gave lectures on his uniformisation theorem for Haken manifolds (\cite{ThL, Th1}).
The theorem states that every atoroidal Haken 3-manifold with its (possibly empty) boundary consisting only of incompressible tori admits a complete hyperbolic metric in its interior.
His proof of this theorem is based on an induction making use of  a hierarchy for Haken manifolds invented by Waldhausen \cite{Wa}, i.e., a system of incompressible surfaces cutting the manifold down to balls, together with Maskit's combination theorem (see for instance \cite[\S VII]{Mas}).

For simplicity, we now focus on the case of closed atoroidal Haken manifolds.
In the last step of the induction, we are in the situation where $N$ is a closed atoroidal Haken manifold obtained from a 3-manifold $M$ with non-empty boundary (without torus components) by pasting  $\partial M$ to itself by an orientation reversing involution.
The induction hypothesis guarantees the existence of a convex cocompact hyperbolic structure on $M$.
There, Thurston used the so-called bounded image theorem to find a convex compact hyperbolic structure on $M$, obtained by quasi-conformally deforming the given hyperbolic structure,  which can be pasted up along $\partial M$ to give a hyperbolic structure on  $N$.

Let us explain the setting in more detail.
Let $M$ be an atoroidal Haken manifold with an even number of boundary components all of which are incompressible.
In the same way as we assumed that $N$ is closed, we assume that no boundary component of $M$ is a torus, for simplicity.
Suppose that there is an orientation-reversing involution $\iota: \partial M \to \partial M$ taking each component of $\partial M$ to another one.
Let $N$ be the closed manifold obtained from $M$ by identifying the points on $\partial M$ with their images under $\iota$.
Suppose moreover that $N$ is also atoroidal.

We assume, as the hypothesis of induction, that $M$ admits a convex compact hyperbolic structure; in other words, that the interior of $M$ is homeomorphic to $\hyperbolic^3/\Gamma$ for a convex cocompact Kleinian group $\Gamma$. 
The space of convex compact hyperbolic structures on $M$, which is not empty by assumption, modulo isotopy is parameterised by $\teich(\partial M)$, as can be seen in the works of Ahlfors, Bers, Kra, Maskit, Marden and Sullivan.
From each convex compact hyperbolic structure on $M$, by taking the covering of $M$ associated with each component $S$ of $\partial M$, we get a quasi-Fuchsian group isomorphic to $\pi_1(S)$, and by considering the second coordinate of the parameterisation $\teich(S) \times \teich(\bar S)$ of the quasi-Fuchsian space, we obtain a map from $\teich(\partial M)$ to $\teich(\bar S)$, where $\teich(\bar S)$ denotes the Teichm\"{u}ller space of $S$ with orientation reversed.
By considering this for every component of $\partial M$, we get a map $\sigma: \teich(\partial M) \to \teich(\bar\partial M)$  called the skinning map, where $\teich(\bar \partial M)$ denotes the product of $\teich(\bar S)$ for the components $S$ of $\partial M$.
Since $\iota$ is orientation-reversing, it induces a homeomorphism $\iota_* \colon \teich(\bar\partial M) \to \teich(\partial M)$.

Then the bounded image theorem can be stated as follows.
\begin{theorem}
\label{main}
Suppose that $M$ is a compact (orientable) atoroidal Haken manifold having an even number of  boundary components all of which are incompressible and none of which are tori, and assume that $M$ is not homeomorphic to an $I$-bundle over a closed surface.
Assume moreover that $M$ admits a
convex compact hyperbolic structure.
Suppose that there is an orientation reversing involution $\iota: \partial M \to \partial M$ taking each component of $\partial M$ to another component, and that by pasting each component of $\partial M$ to its image under $\iota$, we get a closed atoroidal manifold $N$.
Then there exists $n \in \naturals$ depending only on the topological type of $M$ such that the image of $(\iota_* \circ \sigma)^n$ is bounded (\ie precompact) in $\teich(\partial M)$.
\end{theorem}

There are several expository papers and books on Thurston's uniformisation theorem  (\cite{Mo, Ot, Kap} among others).
In all of them, a weaker version of the bounded image theorem called the bounded orbit theorem, which is sufficient for the proof of the uniformisation theorem, was proved and used, instead of this original one.

Up to now, no complete proof of the bounded image theorem as stated above was known.
Kent  \cite{Kent} gave a proof of this theorem under the assumption that $M$ is acylindrical, in which case the deformation space of hyperbolic structures on $M$ is compact.

The purpose of this paper is to give a proof of the original bounded image theorem.
Our argument relies on recent progress in Klenian group theory, in particular, the embedding of partial cores in the geometric limit from \cite{BBCM}, the relation between the presence of short curves and their relative positions and the behaviour of ends invariant from \cite{BBCM-gt}, and criteria of convergence/divergence given in \cite{BBCL}.\\

\begin{generalization}
We can generalise the previous statement by considering the case when $M$ may have torus boundary components and the domain of the involution $\iota$ is a subsurface of $\partial M$. 
Let $\partial_0 M$ be the union of an even number of connected components with negative Euler characteristics of $\partial M$ and let $\iota: \partial_0 M \to \partial_0 M$ be an orientation reversing involution which maps each component of $\partial_0 M$ to another component. 
We assume $\iota$ to be the identify on $\partial M \setminus \partial_0 M$.
Let $\partial_T M$ be the union of torus components of $\partial M$, and $\partial_1 M$ the union of the components of $\partial M \setminus \partial_0 M$ which are not tori.

Now we consider an analogue of the skinning map in this setting.
For each component $S$ of $\partial_1M$, let $\gamma^S$ be the union of simple closed curves homotopic in $M$ into $\partial_TM$.
By Jaco-Shalen-Johannson theory, we see that $\gamma^S$ consists of disjoint simple closed curves.
We let $\gamma^M$ be the union of the $\gamma^S$ for all components $S$ of $\partial M \setminus \partial_TM$.
To consider the last step of the induction in the proof of the uniformisation theorem, we need to consider geometrically finite hyperbolic structures on $M$ in which both each component of $\gamma^M$ and each component of $\iota(\gamma^M)$ are parabolic.

Let $P$ be the union of disjoint annular neighbourhoods of $\gamma^M$ and $\partial_T M$, and set $Q=\iota(P)$.
Let $QH(M, Q)$ be the space of geometrically finite hyperbolic structures on $M$ modulo isotopy whose parabolic locus is precisely $Q$.
The theory of Ahlfors, Bers, Kra, Maskit, Marden and Sullivan gives a parametrisation $q \colon \teich(\partial M \setminus Q) \to QH(M,Q)$.
Let $\teich(\bar \partial M \setminus P)$ be the Teichm\"{u}ller space of $\partial M \setminus P$ with orientation reversed.
For each component $S$ of $\partial M$, by taking a covering associated to $\pi_1(S)$, and considering the parametrisation of the quasi-conformal deformation space  of the obtained geometrically finite Kleinian surface group, we have a map $\sigma^S \colon QH(M, Q) \to \teich(S \setminus Q)\times \teich(\overline{S \setminus P})$, where the latter denotes the Teichm\"{u}ller space of $S \setminus P$ with orientation reversed.

For each fixed $h \in \teich(\partial_1 M \setminus Q)$,
we define the restricted skinning map $\sigma_h \colon \teich(\partial_0 M \setminus Q) \to \teich(\overline{\partial_0 M \setminus P})$ by taking $g \in \teich(\partial_0 M \setminus Q)$ to the structure corresponding to $q(g,h) \in QH(M,Q)$ and taking the covering associated with $S$ for each component $S$ of $\partial_0 M$ to obtain a point in $\teich(\overline{S \setminus P})$ as above.


\begin{theorem}
\label{more general}
Suppose that $M$ is an (orientable) atoroidal Haken manifold which is not homeomorphic to an $I$-bundle over a closed surface nor a solid torus. 
We let the subsurface $\partial_0 M$ of $\partial M$, the parabolic locus $Q$, the involution $\iota\colon \partial M \to \partial M$ and the skinning map $\sigma_h$ be as defined above.
Then there exists $n \in \naturals$ depending only on $\chi(\partial_0 M)$ such that for any $h\in\teich (\partial_1 M \setminus Q)$ the image of $(\iota_* \circ \sigma_h)^n$ is bounded (\ie precompact) in $\teich(\partial_0 M \setminus Q)$.
\end{theorem}

\end{generalization}

\subsection{Outline}
We are going to find $n$ such that if the image of $(\iota_* \circ \sigma_m)^n$ is unbounded then $N$ contains a non-peripheral incompressible torus, contradicting our assumption. For that purpose we shall use the invariant $m$ introduced in \cite{BBCL}.

Given a simple closed curve $d$ on a closed surface $S$ equipped with a hyperbolic metric $g$, we define $$m(g,d,\mu)=\max\left\{\sup_{\substack{Y : \ d\subset\partial Y}}d_Y(\mu(g),\mu),{1\over \len_{g}(d)}\right\},$$
where $\mu(g)$ is a shortest marking for $(S,g)$, $\mu$ is a full marking, and the supremum of the first term in the maximum is taken over all incompressible
subsurfaces $Y$ of $S$ whose boundaries $\partial Y$ contain  $d$. See \cref{invariant m} for more details.

It is not difficult to see that in the setting of \cref{main},  for a given sequence $\{m_i\}$ in $\teich(\partial M)$, if the sequence $\{\sigma(m_i)\}$ is unbounded, then there is a simple closed curve $d$ such that $m(\sigma(m_i),d,\mu)$ is unbounded (see \cref{no parabolics=bounded}). The core of our argument consists in showing, with the help of arguments from \cite{BBCM-gt} and \cite{BBCL}, that in this situation, there is a simple closed curve $d'\subset\partial M$ such that $\{m(m_i,d', \mu)\}$ is unbounded and that $d\cup d'$ bounds an essential annulus in $M$. Using this argument repeatedly, we build (when $\{(\iota_* \circ\sigma)^n(m_i)\}$ is unbounded) an annulus in $N$ which goes through the interior of $M$ (viewed as a subset of $N$) $n$ times. If $n$ is large enough, this annulus must create an essential torus in $N$, and contradicts the assumption that $N$ is atoroidal.

Although this is the overall logic of the proof, in the following sections, we shall present the main steps in a different order. 
After setting up some preliminary definitions in \cref{prelim}, we shall discuss the topological part of the proof  in \cref{topology}. 
First we show that we can add some assumptions on the topology of $M$ which will simplify the arguments later on. 
Next, we study incompressible surfaces on $\partial M$ which can be extended multiple times through the characteristic submanifold of $M$ when it is viewed as a submanifold of $N$. 
This will give us an integer $n$ which appears in  \cref{main}. In  \cref{conv-div} we shall discuss the relation between the behaviour of the invariant $m$ defined above, and the convergence and divergence of Kleinian groups. In  \cref{main prop} we shall prove our key proposition, and obtain the curve $d'$ mentioned above. Finally in  \cref{conclusion} we shall put these pieces together to prove our main theorem.


%

\section{Preliminaries}	\label{prelim}
\subsection{Haken manifolds and characteristic submanifolds}
\label{3-manifold}
An orientable irreducible compact 3-manifold which contains a non-peripheral incompressible surface is called a {\em Haken} manifold.
We note that a compact irreducible 3-manifold with non-empty boundary is always Haken except for a 3-ball.
We say that a Haken manifold is {\em atoroidal} when it does not contain a non-peripheral incompressible torus, and {\em acylindrical} when it does not contain a non-peripheral incompressible annulus.
By the torus theorem for Haken manifolds (\cite{Wa2, JaSh, Joh}), the former condition of the atoroidality is equivalent to the one that every monomorphism from $\integers \times \integers$ into the fundamental group is peripheral, \ie is conjugate to the image of the fundamental group of a boundary component.

The Jaco-Shalen-Johannson theory \cite{JaSh,Joh} tells us that in a Haken manifold, incompressible tori and incompressible annuli can stay only in a very restricted place.
Let us state what the theory says in the case when a Haken manifold $M$ is atoroidal and boundary-irreducible.

For an orientable atoroidal Haken boundary-irreducible 3-manifold $M$, there exists a 3-submanifold $X$ of $M$ each of whose components is one of the following and which satisfies the following condition:
\begin{enumerate}[(a)]
\item An $I$-bundle whose associated $\partial I$-bundle coincides with its intersection with $\partial M$.
Such a component is called a {\em characteristic $I$-pair}.
\item A solid torus $V$ such that $V \cap \partial M$ consists of annuli which are incompressible on both $\partial V$ and $\partial M$.
When $V \cap \partial M$ is connected, it winds around the core curve of $V$ more than once.
\item
A thickened torus $S^1 \times S^1 \times I$ at least one of whose boundary components lies on a component of $\partial M$. 
\end{enumerate}
Every properly embedded essential annulus (i.e. an incompressible annulus which is not homotopic into the boundary) is properly isotopic into $X$, and no component of $X$ is properly isotopic into another component.

Such $X$ is unique up to isotopy, and is called the {\em characteristic submanifold} of $M$.
We note that in the case when $M$ has no torus boundary component, which is the assumption of our main theorem, a component of the last type (c) does not appear.

Thurston's celebrated uniformisation theorem for Haken manifolds says that every atoroidal Haken manifold whose boundary consists of incompressible tori admits a hyperbolic structure of finite volume.
More generally, he proved that every atoroidal Haken manifold, including the case when it has non-torus boundary components, admits a (minimally parabolic) convex hyperbolic structure of finite volume.
The term \lq convex hyperbolic structure' will be explained in the following subsection.
	
%

\subsection{Kleinian groups and their deformation spaces}
A Kleinian group is a discrete subgroup of $\pslc$.
In this paper, we always assume Kleinian groups to be torsion free, and finitely generated except for the case when we talk about geometric limits.
For a Kleinian group $\Gamma$, we can consider the complete hyperbolic 3-manifold $\hyperbolic^3/\Gamma$.
The {\em convex core} of $\hyperbolic^3/\Gamma$ is the smallest convex submanifold that is a deformation retract.
The Kleinian group $\Gamma$ and the corresponding hyperbolic 3-manifold $\hyperbolic^3/\Gamma$ are said to be {\em geometrically finite} when the convex core of $\hyperbolic^3/\Gamma$ has finite volume.
In particular, $\hyperbolic^3/\Gamma$ is said to be {\em convex compact}, and $\Gamma$ to be {\em convex cocompact} if the convex core is compact.
We also say that $\Gamma$ is {\em minimally parabolic} when every parabolic element in $\Gamma$ is contained in a rank-2 parabolic subgroup.
Any convex cocompact Kleinian group is automatically minimally parabolic since it does not have parabolic elements.

%
A 3-manifold $M$ is said to have a  hyperbolic structure when $\Int M$ is homeomorphic to 
$\hyperbolic^3/\Gamma$ for a  Kleinian group $\Gamma$, and we regard the pull-back of the hyperbolic metric to $\Int M$ as a hyperbolic structure on $M$.
In particular if $\Gamma$ is taken to be geomerically finite or convex cocompact, we say that $M$ has a geometrically finite or convex compact hyperbolic structure.
If $M$ admits a hyperbolic structure, then $M$ must be atoroidal.

The set of hyperbolic structures on $M$ modulo isotopy, which we denote by $\AH(M)$, can be identified with a subset of the set of faithful discrete representations of $\pi_1(M)$ into $\pslc$ modulo conjugacy.
We put on  $\AH(M)$ a topology induced from the weak topology on the representation space.
We regard an element of $\AH(M)$ both as a hyperbolic structure on $M$ and as a representation of $\pi_1(M)$ into $\pslc$ depending on the situation.

A Kleinian group $G$ is said to be a {\em quasi-conformal deformation} of another Kleinian group $\Gamma$ if there is a quasi-conformal homeomorphism $f \colon \hat \complexes \to \hat \complexes$ such that $G= f \Gamma f^{-1}$ as M\"obius transformations on $\hat \complexes$.
When $G$ is a quasi-conformal deformation of $\Gamma$, there is a diffeomorphism from $\hyperbolic^3/\Gamma$ to $\hyperbolic^3/G$ preserving the parabolicity in both directions, which induces an isomorphism between the fundamental groups coinciding with the isomorphism given by the conjugacy $G=f\Gamma f^{-1}$.
We note that a quasi-conformal deformation of geometrically finite (resp. convex cocompact, minimally parabolic geometrically finite) group is again geometrically finite (resp. convex cocompact, minimally parabolic geometrically finite).

Let $M$ be a compact 3-manifold admitting a minimally parabolic geometrically finite hyperbolic structure $m$.
Let $\QH(M)$ denote the set of all minimally parabolic geometrically finite hyperbolic structures on $M$ modulo isotopy, which is regarded as a subset of $\AH(M)$.
Marden \cite{Mar} showed that every minimally parabolic geometrically finite hyperbolic structures on $M$ is obtained as a quasi-conformal deformation of $m$.
Therefore we call $\QH(M)$ the {\em quasi-conformal deformation space}.
Furthermore, if $\partial M$ is incompressible, combined with the work of Ahlfors, Bers, Kra, Maskit and Sullivan, there is a parameterisation $q \colon \teich(\partial M) \to \QH(M)$, where $\teich(\partial M)$ denotes the Teichm\"uller space of $\partial M$, \ie the direct product of the Teichm\"uller spaces of the components of $\partial M$.
We shall refer to this map as the {\em Ahlfors-Bers map}.

\begin{generalization}
In the setting of \cref{more general}, we need to consider 

\end{generalization}

In the case when $M$ is homeomorphic to $S \times [0,1]$ for a closed oriented surface $S$, the deformation spaces $\AH(M), \QH(M)$ are denoted by $\AH(S), \QF(S)$ respectively.
The quasi-conformal deformation space $\QF(S)$ consists of quasi-Fuchsian representations of $\pi_1(S)$, \ie quasi-conformal deformations of a Fuchsian representation, and is therefore called the {\em quasi-Fuchsian space}.
The Ahlfors-Bers map can be expressed as $qf: \teich(S) \times \teich(\bar S) \to \QF(S)$, where the second coordinate $\teich(\bar S)$ denotes the Teichm\"uller space of $S$ with orientation reversed, which is a more natural way for parametrisation since the boundary component $S\times \{1\}$ has the opposite orientation from the one given on $S\times \{0\}$ if we identify them with $S$ by dropping the second factor.

Now, let $M$ be an atoroidal Haken 3-manifold with non-empty incompressible boundary which does not contain a torus. 
Suppose that $M$ has a convex compact hyperbolic metric $m$, and let $S$ be a component of $\partial M$. Take a covering of $M$ associated with $\pi_1(S) \subset \pi_1(M)$, and lift the hyperbolic structure $m$ to the hyperbolic structure $\tilde m$ on $S \times [0,1]$.
It is known (see \cite[Proposition 7.1]{Mo}) that the lifted structure $\tilde m$ is also convex cocompact, hence can be regarded as an element of $\QF(S)$.
Therefore $\tilde m$ in turn corresponds to a point $(g_S(m), h_S(m))$ in $\teich(S) \times \teich(\bar S)$.
Let $S_1, \dots , S_k$ be the components of $\partial M$ that are not tori, and we consider the point $h_{S_i}(m) \in \teich(\bar S_i)$ for each $i=1, \dots ,k$.
We define $\teich(\bar \partial M)$ to be $\teich(\bar S_1) \times \dots \times \teich(\bar S_k)$.
The map taking $g \in \teich(\partial M)$ to $(h_{S_1}(q(g)), \dots, h_{S_k}(q(g))) \in \teich(\bar \partial M)$ is called the {\em skinning map}, which we shall denote by $\sigma$.

\begin{generalization}
Also in the setting of \cref{more general}, we consider a minimally parabolic geometrically finite hyperbolic structure $m$ on $M$.
The boundary of $M$ is divided into three union of components, $\partial_0 M$, $\partial_1 M$, and the union of the torus components.
Let $S$ be a component of $\partial_0 M$.
Then $m$ lifts to a geometrically finite hyperbolic structure $\tilde M$ on $S \times [0,1]$.
We note that $\tilde m$ may not be convex cocompact, since there might be a simple closed curve on $S$ which is homotopic into a torus boundary component in $M$.
\end{generalization}
\subsection{Curve complexes and projections}
\label{cc}
Let $S$ be a connected compact orientable surface possibly with boundary, satisfying $\xi(S)=3g+n \geq 4$ where $g$ denotes the genus and $n$ denotes the number of the boundary components.
The {\em curve complex} $\cc(S)$ of $S$ with $\xi(S) \geq 5$ is a simplicial complex whose vertices are  isotopy classes of non-peripheral, non-contractible simple closed curves on $S$ such that  $n+1$ vertices span an $n$-simplex when they are represented by pairwise disjoint simple closed curves.
In the case when $\xi(S)=4$, we define $\cc(S)$ to be a graph whose vertices are isotopy classes of simple closed curves such that two vertices have smallest possible intersection.
In the case when $S$ is an annulus, we define $\cc(S)$ to be a graph whose vertices are isotopy classes (relative to the endpoints) of non-peripheral simple arcs in $S$ such that two vertices are connected when they can be represented by arcs which are disjoint at their interiors.
Masur-Minsky \cite{MM} proved that $\cc(S)$  is Gromov hyperbolic  with respect to the path metric for any $S$.

A {\em marking} $\mu$ on $S$ consists of a pants decomposition of $S$, which is denoted by $\base(\mu)$ and whose components are called base curves, and  a collection $\mathbf t(\mu)$ of simple closed curves, called transversals of $\base(\mu)$, such that each component of $\base(\mu)$ intersects at most one among them essentially.
For two markings  $\mu, \nu$ on $S$ and a subsurface $Y$, we define $d_Y(\mu, \nu)$ to be the distance between $\pi_Y(\base(\mu)\cup \mathbf t (\mu))$ and $\pi_Y(\base(\nu)\cup \mathbf t(\nu))$, where the projection $\pi_Y \colon \cc(S) \to \mathfrak{P}(\cc(Y))$ is obtained by taking the intersection of curves on $S$ with $Y$ and connecting the endpoints by arcs on $\Fr Y$ when the intersection contains arcs.
In \cite{MM2}, a marking defined as such is called clean.
In this paper, we only consider clean markings.
A marking is called {\em full} when every base curve has a transversal.
In general, for two sets of simple closed curves $a, b$ and a subsurface $Y$ of $S$, we define $d_Y(a,b)$ to be the distance in $\cc(Y)$ between $\pi_Y(a)$ and $\pi_Y(b)$ provided that both of them are non-empty.
If one of them is empty, the distance is not defined.
For a point $m$ in $\teich(S)$, its shortest marking, which is a full marking and is denoted by  $\mu(m)$, has  a shortest pants decomposition of $(S,m)$ as $\base(\mu(m))$, and $\mathbf t(\mu(m))$ consisting of shortest transversals, one for each component of $\base(\mu(m))$.
When we talk about the distance $d_Y$ between two points in $\teich(S)$ or between a point in $\teich(S)$ and a marking, we identify points $m \in \teich(S)$ with $\mu(m)$.

\subsection{Geometric limits and compact cores}
\label{geom limit}
Let $M$ be an atoroidal boundary-irreducible Haken 3-manifold.
Let $\{\rho_i\}$ be a sequence of faithful discrete representations of $\pi_1(M)$ into $\pslc$.
We define a {\em geometric limit} of $\{\rho_i(\pi_1(M))\}$ to be a Kleinian group $\Gamma$ such that every element $\gamma$ of $\Gamma$ is a limit of some sequence $\{g_i \in \rho_i(\pi_1(M))\}$, and every convergent sequence $\{\gamma_{i_j} \in \rho_{i_j}(\pi_1(M))\}$ has its limit in $\Gamma$.

Fixing a point $x \in \hyperbolic^3$, and considering its projections $x_i$ in $\hyperbolic^3/\rho_i(\pi_1(M))$ and $x_\infty$ in $\hyperbolic^3/\Gamma$, the geometric convergence implies the existence of pointed Gromov-Hausdorff convergence of $((\hyperbolic^3/\rho_i(\pi_1(M)), x_i))$ to $(\hyperbolic^3/\Gamma, x_\infty)$.
This latter convergence means that there exist real numbers $r_i$ going to $\infty$, $K_i$ converging to $1$, and $K_i$-bi-Lipschitz diffeomorphisms $f_i$ (called {\em approximate isometries}) between $r_i$-balls $B_{r_i}(\hyperbolic^3/\rho_i(\pi_1(M)),x_i)$ and $B_{K_ir_i}(\hyperbolic^3/\Gamma, x_\infty)$.
Suppose that $\{\rho_i\}$ converges to $\rho_\infty: \pi_1(M) \to \pslc$ as representations and that $\{\rho_i(\pi_1(M))\}$ converges to $\Gamma$ geometrically.
Then, $\rho_\infty(\pi_1(M))$ is a subgroup of the geometric limit $\Gamma$.

For an open irreducible 3-manifold $V$ with finitely generated fundamental group, a compact 3-dimensional submanifold $C$ in $V$ is called a {\em compact core} when the inclusion induces an isomorphism between their fundamental groups.
The existence of compact cores was proved by Scott \cite{Sc}.
The case which interests us is when $V$ is a hyperbolic 3-manifold.

Let $\hyperbolic^3/G$ be a hyperbolic 3-manifold associated with a finitely generated, torsion free Kleinian group $G$.
By Margulis's lemma, there is a positive constant $\varepsilon_0$ such that the set of points of $\hyperbolic^3/G$ where the injectivity radii are less than $\varepsilon_0$ consists of a finite disjoint union of tubular neighbourhoods of closed geodesics of length less than $\varepsilon_0$, called {\em Margulis tubes}, and {\em cusp neighbourhoods} each of which is stabilised by a maximal parabolic subgroup of $G$, and whose quotient by its stabiliser is homeomorphic to $S^1 \times \reals^2$ when the stabiliser has rank 1, and to $S^1 \times S^1 \times \reals$ when the stabiliser has rank 2.
The former cusp neighbourhood is called a $\integers$-cusp neighbourhood, and the latter a torus cusp neighbourhood.
The union of the cusp neighbourhoods is called the cuspidal part of $\hyperbolic^3/G$.
The complement of the cuspidal part is called the {\em non-cuspidal part} and is denoted by $(\hyperbolic^3/G)_0$.
Each boundary component of $(\hyperbolic^3/G)_0$ is either an open annulus or a torus.
By the relative compact core theorem by McCullough \cite{Mc}, there is a compact core $C_G \subset (\hyperbolic^3/G)_0$ such that for each boundary component $B$ of $(\hyperbolic^3/G)_0$, the intersection $C_G \cap B$ is a core annulus when $B$ is an open annulus, and is the entire $B$ when $B$ is a torus.
We call such a compact core a {\em relative compact core} of $(\hyperbolic^3/G)_0$.


Let $p \colon \hyperbolic^3/\rho_\infty(\pi_1(M)) \to \hyperbolic^3/\Gamma$ be the covering map associated with the inclusion of  $\rho_\infty(\pi_1(M))$ into the geometric limit $\Gamma$.
Let $C$ be a relative compact core of $(\hyperbolic^3/\rho_\infty(\pi_1(M)))_0$.
Suppose that $\hyperbolic^3/\Gamma$ has a torus cusp neighbourhood $T$.
We say that $\hyperbolic^3/\rho_\infty(\pi_1(M))$ {\em wraps around} $T$ when $p|C$ is homotoped to an immersion which goes around $T$ non-trivially, and hence cannot be homotoped to an embedding.

\section{Topological features}	\label{topology}

\subsection{Coverings}
In this subsection, we shall show that  to prove \cref{main}, 
 we can put an extra assumption that all the characteristic  $I$-pairs of $M$ are product bundles.
 
We consider an atoroidal Haken manifold as is given in \cref{main}. 
Let $p \colon \tilde{M}\to M$ be a finite-sheeted regular covering.
Then $p$ induces the covering map between the boundaries $p_\partial \colon \partial \tilde M \to \partial M$.
This map induces a proper embedding between Teichm\"{u}ller spaces, $p_\partial^* \colon \teich(\partial M) \to \teich(\partial \tilde M)$ which is obtained by pulling back the conformal structures by $p_\partial$.
Also the involution $\iota$ lifts to an orientation-reversing involution $\tilde \iota \colon \partial \tilde M \to \partial \tilde M$ taking each component to another one.
Since $\tilde M$ is also an atoroidal boundary-irreducible Haken manifold, we can consider the skinning map $\tilde \sigma \colon \teich(\partial \tilde M) \to \teich(\bar{\partial} \tilde M)$.

\begin{generalization}
Also in the setting of \cref{more general}, for a finite-sheeted covering $p \colon \tilde M \to M$, we have a covering map $p_\partial \colon \partial_0 \tilde M \to \partial_0 M$, where $\partial_0 \tilde M=p^{-1}(\partial_0 M)$.
In the same way as above, we can consider the skinning map $\tilde \sigma_h \colon \teich(\partial_0 \tilde M) \to \teich(\bar{\partial}_0 \tilde M)$ lifting $\sigma_h$ for every $h \in \teich(\partial_1 M)$.
\end{generalization}

\begin{lemma}
\label{covering}
If $(\tilde \iota \circ \tilde \sigma)^n$ has bounded image for some $n\in \naturals$, then so does $(\iota \circ \sigma)^n$.

\begin{generalization}
In the same way, for every $h \in \teich(\partial_1M)$,  if $(\tilde \iota \circ \tilde \sigma_h)^n$ has bounded image for $n$, then so does $(\iota \circ \sigma_h)^n$.
\end{generalization}
\end{lemma}
\begin{proof}
The map $p^*_\partial$ properly embeds $\teich(\partial M)$ into $\teich(\partial \tilde M)$.
By the definition of the maps $\tilde \sigma$ and $\tilde \iota$, we have $p^*_\partial \circ (\iota \circ \sigma)=(\tilde \iota \circ \tilde \sigma) \circ p^*_\partial$, and hence $p^*_\partial \circ (\iota \circ \sigma)^n=(\tilde \iota \circ \tilde \sigma)^n \circ p^*_\partial$.
Therefore, if the image of $(\tilde \iota \circ \tilde \sigma)^n$ is bounded, the properness of $p^*_\partial$ implies that the image of $(\iota \circ \sigma)^n$ must also be bounded.
\begin{generalization}
The same argument works also for the second statement.
\end{generalization}
\end{proof}

This result allows us to work on manifolds with topological features that will make the arguments simpler.

\begin{lemma}
	\label{trivial pairs}
	Let $M$ be an orientable atoroidal Haken manifold with incompressible boundary.
	\begin{generalization}
		or in \cref{more general}.
	\end{generalization}
	Then there is a double covering of $M$ all of whose characteristic $I$-pairs are product $I$-bundles.
\end{lemma}
\begin{proof}
	Let $W_1, \dots , W_p$ be the characteristic $I$-pairs of $M$ that are twisted $I$-bundles.
	Take their double coverings $\tilde W_1, \dots, \tilde W_p$ corresponding to the orientation double coverings of the base surfaces.
	For each $W_j$ among $W_1, \dots , W_p$,  its frontier components (\ie the closures of the components of $\partial W_j \setminus \partial M$) are annuli.
	Each of such annuli has two pre-images in $\tilde W_j$ which are taken to each other by the unique non-trivial covering translation.
	
	Let $C$ be the closure of a component of $M \setminus (W_1 \cup \dots \cup W_p)$.
	Let $A_1, \dots , A_k$ be the components of $C \cap (W_1\cup \dots \cup W_p)$, which are annuli on $\cup_{l=1}^p \overline{\partial W_l \setminus \partial M}$.
	We prepare two copies $C^+$ and $C^-$ of $C$.
	Each $A_j$ among $A_1, \dots A_k$, which is contained some $W_i$ among $W_1, \dots , W_p$, has two lifts $A_j^+$ and $A_j^-$ in $\tilde W_i$.
	Now we identify the copy of $A_j$ in $C^+$ to $A_j^+$ in $\tilde W_i$ and the one in $C^-$ to $A_j^-$ in $\tilde W_i$ for each annulus among $A_1, \dots A_k$.
	We repeat the same procedure for every component $C$ of $M \setminus (W_1 \cup \dots \cup W_p)$, and get a manifold $\tilde M$, which will turn out to be a double cover of $M$ as shown below.
	
	Define a homeomorphism $t\colon \tilde M \to \tilde M$ to be the covering translation on each $\tilde W_i$ and  the map taking $C^\pm$ to $C^\mp$ preserving the identification with $C$ for each of $C^\pm$.
	It is clear from the definition that this homeomorphism $t$ is a free involution.
	By taking the quotient of $\tilde M$ under $\langle t \rangle \cong \integers_2$, we get a manifold naturally identified with $M$.
	Thus we see that $\tilde M$ is a double cover of $M$.
	Since $\tilde W_i$ is a product $I$-bundle and all of the characteristic $I$-pairs contained in $C^\pm$ are product bundles by our definition of $W_1, \dots , W_p$, we see that $\tilde M$ is a double cover as desired. 
\end{proof}

For some of our arguments we will need a stronger assumption than having only product bundles:

\begin{definition}
Let $M$ be a compact orientable Haken $3$-manifold with incompressible boundary. We say that $M$ is {\em strongly untwisted} if and only if:
\begin{enumerate}[(A)]
	\item Every characteristic $I$-pair is a product bundle.
	\item For any characteristic $I$-pair $\Xi$ and any simple closed curve $d\subset\partial M$, the simple closed curve $d$ can be homotoped on $\partial M$ into at most one component of $\Xi\cap\partial M$.
\end{enumerate}

\end{definition}

We are going to construct a cover with the properties (A) and (B) above. 
In order to do that, we need to examine how characteristic $I$-pairs are attached to  other components of the characteristic submanifold. 
In the following proof of \cref{strongly untwisted}, it will turn out that there are two situations ((a) and (b) below) where the second condition of \lq strong untwistedness' breaks down.

\begin{lemma}
\label{strongly untwisted}
Let $M$ be a compact orientable atoroidal Haken manifold with incompressible boundary. Then there is a finite-sheeted regular covering of $M$ which is strongly untwisted.
\end{lemma}
\begin{proof}
By Lemma \ref{trivial pairs}, we have a double covering all of whose characteristic $I$-pairs are product $I$-bundles.
Therefore, we may assume that $M$ satisfies the first condition (A) of \lq strong untiwistedness', and we shall construct a covering satisfying the second condition.

To construct such a covering, let us analyse how this second condition (B) can fail to hold. 
Let $d\subset\partial M$ be a simple closed curve, and let $W$  a characteristic $I$-pair. 
Since no components of $W\cap\partial M$ are  annuli, $d$ can be homotoped on $\partial M$ into at most two components of $W \cap \partial M$. 
Furthermore, if $d$ can be homotoped into two such components, then $d$ lies (up to isotopy on $\partial M$) on a component $T_j$ (characteristic solid torus or a characteristic thickened torus) of $T$, and  there are two possibilities:
(a)  $T_j \cap \partial M$ is an annulus when $T_j$ is a solid torus, and is the union of an annulus and a torus when $T_j$ is a thickened torus; or
(b) $d$ separates two consecutive components of $T_j\cap\overline{M\setminus T}$ both lying in the same characteristic $I$-pair. 
We shall show that we can take a finite-sheeted covering of $M$ so that neither (a) nor (b) can happen.

First, we consider the condition  (a).
Let $T_j$ be a component of $T$ such that $T_j \cap \partial M$ is an annulus (and $T_j$ is a solid torus) or the union of an annulus and a torus (when $T_j$ is a thickened torus).
This implies that $T_j \cap \overline{M \setminus T_j}$ is  connected, hence is an annulus, which we denote by $A$.
Since $T_j$ is a characteristic solid torus or characteristic thickened torus, the annulus $A$ is essential, and hence is not homotopic to $T_j \cap \partial M$ fixing the boundary.
Then, we can choose a simple closed curve $\alpha$, which is not contractible in $M$, on the component of $\partial T_j$ on which $d$ lies so that both $\alpha \cap A$ and $\alpha \cap \partial M$ are connected, \ie arcs.
Since $\pi_1(T_j)$ is either $\integers$ or $\integers \times \integers$, we can take a $k$-sheeted cyclic covering $\tilde T_j$ of $T_j$ so that $\alpha$ cannot be lifted homeomorphically, whereas the annulus $A$ is homeomorphically lifted.
(For instance, in the case when $\pi_1(T_j) \cong \integers$, we choose $k$ which is coprime with the element represented by $\alpha$.)
Then the preimage of  the annulus $A$ is $k$ copies of $A$, which we denote by $A_1, \dots, A_k$.
Let $C$ be $\overline{M \setminus T_j}$.
We prepare $k$ copies of $C$,  which we denote by $C_1, \dots , C_k$.
By pasting $C_j$ along $A_j$ to $\tilde T_j$, we can make a $k$-sheeted cyclic covering of $M$ in which $\tilde T_j$ does not satisfy the condition (a).
If there is another component $T_{j'}$ of $T$ with the condition (a), we repeat the same process for all the $k$ lifts of $T_{j'}$ at the same time.
Repeating the process, we get a finite-sheeted covering of $M$ in which there is no characteristic solid torus or a characteristic thickened torus with the condition (a).
We use the same symbol $M$ and $T$ for this finite-sheeted covering, abusing the notation.

Now we turn to  the condition (b). We consider three colours named the colour $0$, $1$ and $2$ and we choose a colour for each annulus of $T\cap \overline{M \setminus T}$ so that, on $\partial T$, no two consecutive annuli have the same colour. We take three copies of each component of $T$ and of $\overline{M \setminus T}$ which we name the lift $0$, $1$ and $2$. Consider a component $U$ of $T$, a component $V$ of $\overline{M \setminus T}$ and an annulus $E\subset U\cap V$ with the colour $k \in \{0,1,2\}$. For every $j\in\{0,1,2\}$, we glue the lift $j$ of $V$ to the lift $(j+k)$ mod $3$ of $U$ along the appropriate lifts of $E$. Using the same construction for each component of $T \cap \overline{M \setminus T}$, we get a triple cover $\hat M$ of $M$ in which any two consecutive components of $\hat T\cap\overline{\hat M\setminus\hat T}$ lie in different components of $\overline{\hat M\setminus\hat T}$. In particular there is no characteristic solid torus or characteristic thickened torus in $\hat M$ for which the condition (b) holds.

Thus, we have shown that by taking a finite-sheeted covering, we can make both of the situations (a) and (b) disappear, which means, as we saw above, that the covering is strongly untwisted.
\end{proof}

{\em \cref{covering,strongly untwisted} show that to prove \cref{main}, we have only to consider the case when $M$ is strongly untwisted.}


\subsection{Vertically extendible surfaces}
\label{sec: ext}
Let $M$ be an atoroidal Haken manifold as in \cref{main}.
Let $X$ be the characteristic submanifold of $M$.
Assume that every $I$-bundle in $X$ is a product $I$-bundle.

\begin{definition}
\label{extendible}
Given an incompressible subsurface $F\subset \partial M$, we say that $F$ is {\em one-time vertically extendible} if there is an incompressible surface $F^1\subset\partial M$
\begin{generalization}
(assuming that $F^1$ is disjoint from the torus components of $\partial M$ in the setting of \cref{more general})
\end{generalization}
 and an essential $I$-bundle $V_F\subset M$ with $V_F\cap\partial M=F\cup F^1$ 
and $F^1\subset\partial X$ up to isotopy. We call $F^1$ a {\em first elevation} of $F$.
\end{definition}


It follows from the definition of characteristic submanifold that there is an isotopy which takes $V_F$ into the characteristic submanifold $X$. From now on, we assume that if $F$ is one-time vertically extendible then $F\subset X$ and $V_F\subset X$.

We note  solid torus
\begin{generalization}
and thickened torus
\end{generalization} 
components in $X$
may add some complications in the case when $F$ is an annulus.
If $F$ is contained in such a component of $X$, there may be more than one possible first elevation (even up to isotopy) and the $I$-bundles corresponding to two disjoint annuli may intersect (even up to isotopy).\\

We now define multiple elevations by induction.
\begin{definition}
\label{extendibility}
Given an incompressible subsurface $F$ in $\partial M$ and $n\geq 2$, we say that $F$ is {\em $n$-time vertically extendible}   if there is an essential surface $F^1\subset\partial M$
\begin{generalization}
not contained in a torus boundary component
\end{generalization}
 and an essential $I$-bundle $V_F\subset M$ with $V_F\cap\partial M=F\cup F^1$ and $\iota(F^1)$ is $(n-1)$-time vertically extendible. An $(n-1)$-th elevation $F^n$ of $\iota(F^1)$ is defined to be an $n$-th elevation of $F$.
\end{definition}

We say that two multi-curves $c,d\subset\partial M$ {\em intersect minimally} if for every multicurves $c',d'$ homotopic to $c$ and $d$ respectively, $\sharp\{c\cap d\}\leq\sharp\{c'\cap d'\}$. Let $F,G\subset X\cap\partial M$ be two incompressible surfaces. We say that $F$ and $G$ {\em intersect minimally} if $\partial F$ intersects $\partial G$ minimally.

\begin{lemma}	\label{union cobounded}
Let $F,G\subset\partial M$ be connected incompressible subsurfaces which intersect minimally and are not disjoint. If $F$ and $G$ are $n$-time vertically extendible, then so is $F\cup G$.
\end{lemma}

\begin{proof}
If $F$ and $G$ are one-time vertically extendible, as was remarked before, we may assume that $F,G\subset X\cap\partial M$. 
Since they intersect minimally and are not disjoint, they must lie in the same component $H$ of $X\cap\partial M$ which is not an annulus. Then the component $V$ of $X$ containing $H$ is an $I$-bundle, which is a product $I$-bundle by assumption, and can be parametrised as $H\times[0,1]$.
	
Then, by moving $F$ and $G$ by isotopies, we have  $V_F=F\times [0,1]\subset H\times[0,1]$ and $V_G=G\times [0,1]\subset H\times[0,1]$, and $F^1=F \times\{0,1\} \setminus F,\  G^1=G\times\{0,1\} \setminus G$. 
Since $F^1$, resp. $G^1$, lies in the component of $X\cap\partial M$ which does not contain $F$ and $G$, $F^1$ and $G^1$ lie in the same component of $X\cap \partial M$.
Therefore  $F^1 \cup G^1$ lies in $X\cap\partial M$. Thus we have proved that if $F$ and $G$ are one-time vertically extendible then $F\cup G$ is also one-time vertically extendible and  $F_1\cup G_1$ is its first elevation.
	
The case of $n>1$ follows by induction.
\end{proof}

\begin{corollary}
For any natural number $n$, there is  an (possibly empty) incompressible surface $\Sigma^n\subset X\cap\partial M$ such that each component of $\Sigma^n$ is $n$-time vertically extendible, no component of $\Sigma^n$ can be isotoped on $\partial M$ into another component of $\Sigma^n$ and every  $n$-time vertically extendible surface can be isotoped into $\Sigma^n$.
\end{corollary}

\begin{proof}
If there is no surface that is $n$-time vertically extendible, we set $\Sigma^n$ to be $\emptyset$. 
Otherwise, let $\Sigma \subset X$ be an $n$-time vertically extendible incompressible surface.
 If every $n$-time vertically extendible surface can be isotoped into $\Sigma$, we are done, by taking $\Sigma^n=\Sigma$. 
	
Otherwise, there is an $n$-time vertically extendible surface $F$ which cannot be isotoped into $\Sigma$. 
Moving $F$ by an isotopy we can assume that $F$ intersects $\Sigma$ minimally. 
By Lemma \ref{union cobounded}, each connected component of $F\cup \Sigma$ is $n$-time vertically extendible, and we replace $\Sigma$ with $\Sigma \cup F$, and call this enlarged surface $\Sigma$. 
We repeat this operation as long as there is an $n$-time vertically extendible surface which cannot be isotoped into $\Sigma$.
 Every time we add a surface, either we decrease the Euler characteristic of $\Sigma$ or we add a disjoint annulus which cannot be isotoped into $\Sigma$. Hence this process must terminate after finitely many steps. The final resulting surface is $\Sigma^n$.	
\end{proof}

Since an $n$-time vertically extendible surface is $m$-vertically extendible for any $m\leq n$, we have $\Sigma^n\subset \Sigma^m$ up to isotopy.\\

In the next lemma we show that, when $N$ is atoroidal, $M$ cannot contain an $n$-time extendible surface for sufficiently large $n$. In the last section, this result will lead us to the constant $n$ of Theorem \ref{main}.


\begin{lemma}
\label{non-extendible}
There is $L$ depending only on the topological type of $\partial M$ such that if there is an $L$-time vertically extendible surface, then $N$ is not atoroidal.
\end{lemma}

\begin{proof}
Letting $g$ denote the genus of $\partial M$, we set  $K=3g-3$,  which is the number of curves in a pants decomposition of $\partial M$. 
Since no components of $\Sigma^n$ can be isotoped into another component, $\partial \Sigma^n$ has at most $2K$ boundary components. Using this observation, we show in the following claim that $\Sigma^{n+2K}$ must be a proper subsurface of  $\Sigma^n$ even up to isotopy.

\begin{claim}\label{long bundle}
For any  $n\in\naturals$, if $\Sigma ^{n}$ is non-empty  and  any component of $\Sigma^{n}$ can be isotoped into $\Sigma^{n+2K}$, then $N$ cannot be atoroidal.
\end{claim}

\begin{proof}
Suppose that  $\Sigma^{n}\neq\emptyset$, and that any component of $\Sigma^{n}$ can be isotoped into $\Sigma^{n+2K}$.
Since $\Sigma^{n+j}$ is contained in $\Sigma^n$ for any $j\geq 0$ up to isotopy as observed above, and no component of $\Sigma^{n+j}$ can be isotoped into another component,  we have then $\Sigma^{n+j}=\Sigma^{n}$ for any $j\leq 2K$ up to isotopy.
Let $F$ be a component of $\Sigma^{n+2K}$ with minimal Euler characteristic, and $F^j$ its $j$-th elevation. 
By definition, $\iota(F^j)$ is $(n+2K-j)$-time vertically extendible for any $j\leq 2K$. 
Therefore $\iota(F^j)$ can be isotoped into $\Sigma^{n}$. Since $\Sigma^{n}=\Sigma^{n+2K}$ and $F$ has minimal Euler characteristic, $\iota(F^j)$ is a component of $\Sigma^{n}$, up to isotopy. In particular $\partial(\iota(F^j))\subset \partial \Sigma^{n}$ up to isotopy. 
	
Let $V^j$ be the $I$-bundle cobounded by $\iota(F^{j-1})$ and $F^j$. 
We note that by definition, $F^j$ and $\iota(F^j)$ are identified in $N$ and that the interior of $V^j$ is embedded in $N$.  
Taking the union of  the $I$-bundles $V^j$  in $N$ for $j\leq 2K$, we get a map $F\times [0,2K]\rightarrow N$ such that $F\times\{j\}$ is sent to $F^j$. 
Let $c$ be a component of $\partial F$. 
The image of the annulus $c\times [0,2K]$ goes $2K+1$ times through $\partial \Sigma^{n}$. Since $\partial \Sigma^{n}$ has at most $2K$ components, there is a component $c'$ of $\partial \Sigma^{n}$ through which $c \times [0,2K]$ goes at least twice. The image of the part of this annulus between two such instances forms a torus $T$ embedded in $N$. Considering the component of $\partial M$ through which $T$ goes, we can construct an infinite cyclic covering of $N$ in which $T$ lifts to an infinite incompressible annulus. It follows that $T$ is incompressible and non-peripheral. 
Hence $N$ is not atoroidal.
\end{proof}

As mentioned before, we have $\Sigma^n\subset \Sigma^m$ for any $m\leq n$. 
Consider monotone increasing indices $n_j$ such that $\Sigma^{n_j+1}$ is smaller than $\Sigma^{n_j}$ in the sense that at least one component of $\Sigma^{n_j}$ cannot be isotoped into $\Sigma^{n_j+1}$. Since no component of $\Sigma^n$ can be isotoped into another component, we have then either $\chi(\Sigma^{n_j+1})>\chi(\Sigma^{n_j})$ or $\Sigma^{n_j+1}$ has fewer connected components than $\Sigma^{n_j}$. 
It follows that there are at most $K$ such $n_j$, namely, there is $J\leq K$ such that for any $n\geq n_J+1$, we have $\Sigma^n=\Sigma^{n+1}$. By Claim \ref{long bundle}, if $n_j-n_{j-1}\geq 2K$ for some $j\leq J$ or if $\Sigma^{n_J}\neq\emptyset$, then $N$ is not atoroidal. 
Since $J\leq K$, we can now conclude the proof just by setting $L=2K^2$.
\end{proof}

\section{Convergence, divergence and subsurface projections}	\label{conv-div}

In this section, we shall review the relations between the invariant $m$ mentioned in the introduction and the convergence and divergence of Fuchsian and Kleinian groups.

\subsection{Subsurface projections and Fuchsian groups}
We first recall the definition of the invariant $m$ from \cite{BBCL}, and see how it controls the behaviour of sequences of Fuchsian groups.

\begin{definition}	\label{invariant m}
	Let $S$ be a (possibly disconnected) closed surface of genus at least $2$
	\begin{generalization}
	or a finite-area hyperbolic surface with punctures
	\end{generalization}
	and $g$ a point in its Teichm\"uller space.
	Regarding $g$ as a hyperbolic structure on $S$, we let  $\mu(g)$ be a shortest marking for $(S,g)$ (See \cref{cc}).
	Although there might be more than one shortest markings, its choice does not matter for our definition and arguments.
	We fix a full and clean marking $\mu$ consisting of a pants decomposition and transversals on $S$ independent of $g$.
	For any essential simple closed curve $d$ on $S$, we define $$m(g,d,\mu)=\max\left\{\sup_{\substack{Y : \ d\subset\partial Y}}d_Y(\mu(g),\mu),{1\over \len_{g}(d)}\right\},$$
	where the supremum of the first term in the maximum is taken over all incompressible
	subsurfaces $Y$ of $S$ whose boundaries $\partial Y$ contain  $d$.
\end{definition}

It follows from \cite[Lemma 5.2]{BBCL} that two curves with unbounded $m$ cannot intersect:

\begin{lemma} \label{lemma:multicurve m unbounded}
Let $\{m_i\}$ be a sequence in $\teich(S)$ and let $c_1,c_2$ be simple closed curve on $S$. If $m(m_i,c_i,\mu)\longrightarrow\infty$ for both $i=1,2$, then $i(c_1,c_2)=0$.
\end{lemma}

\begin{proof}
This is just a special case of \cite[Lemma 5.2]{BBCL} for Fuchisan groups.
We note that the assumption of bounded projections of end invariants is unnecessary in this special case for which end invariants are empty.
\end{proof}

The invariant $m$ is related to the divergence and convergence of a sequence by the following lemma:

\begin{lemma}	\label{no parabolics=bounded}
Let $\mu$ be a full marking on $S$, and let $\{m_i\}$ be a sequence in $\teich(S)$. Then every subsequence of $\{m_i\}$ contains a convergent subsequence if and only if $\{m(m_i,c,\mu)\}$ is bounded for every essential simple closed curve $c$ on $S$.
\end{lemma}

\begin{proof}
Let $\mu_i$ be a shortest marking for $m_i$. It follows from classical results on Fenchel-Nielsen coordinates that any subsequence of $\{m_i\}$ contains a converging subsequence if and only if the sequence $\{\mu_i\}$ is a finite set and $\{\len_{m_i}(\mu_i)\}$ is bounded.
By \cite[Lemma 2.3]{BBCL}, the sequence $\{\mu_i\}$ is infinite if and only if passing to a subsequence, there is  an incompressible subsurface $Y$ such that $d_Y(\mu_i,\mu)\longrightarrow \infty$ (and hence $m(m_i,c,\mu)\longrightarrow\infty$ for any component $c$ of $\partial Y$). 

On the other hand, if the sequence  $\{\mu_i\}$ is constant, then $\len_{m_i}(\mu_i)$ is unbounded if and only if passing to a subsequence, there is a curve $c$ with  $\len_{m_i}(c)\longrightarrow 0$ (and hence $m(m_i,c,\mu)\longrightarrow\infty$). 
\end{proof}

\subsection{Relative convergence of Kleinian groups}

We shall next establish a necessary condition on the invariant $m$ for algebraic convergence on a submanifold. 
We start with a fundamental result.
Thurston proved in \cite{Th3} the following which is the first half of the theorem often referred to as the \lq broken window only' theorem.
We note that the latter half of the broken window only theorem should need some rectification (see \cite{OhH}) but is irrelevant to the present paper.

\begin{theorem}
\label{broken window}
Let $M$ be an atoroidal Haken 3-manifold and $X$ its characteristic submanifold.
Then for any curve $\gamma$ in $M \setminus X$ and any sequence $\{\rho_i \in \AH(M)\}$, the length of the closed geodesic in $\hyperbolic^3/\rho_i(\pi_1(M))$ representing the free homotopy  class of $\rho_i(\gamma)$ is bounded as $i \longrightarrow \infty$.
\end{theorem}

Using arguments from \cite{BBCL}, we establish the following necessary condition for algebraic convergence on a submanifold.

\begin{theorem}
\label{relative convergence}
Let $M$ be an atoroidal Haken boundary-irreducible $3$-manifold all of whose characteristic $I$-pairs are product $I$-bundles.
 Let $\{m_i\}$ be a sequence in $\teich(\partial M)$, and $\{\rho_i:\pi_1(M)\to \PSL\}$  a sequence of representations corresponding to $\{q(m_i)\}$. 
 Let $\mu\subset \partial M$ be a full and clean marking, and  $W\subset M$ a submanifold with \lq paring locus' $P$ which is a union of disjoint non-parallel essential annuli on $\partial W$.
 We assume the following:
\begin{enumerate}[(a)]
		\item The closure  of $W\setminus \partial M$ is a union of essential annuli contained in $P$.
		\item 
		\label{P}
		For any non-contractible simple closed curve $c$ in $P$,  $\len_{\rho_i}(c)$ is bounded as $i \longrightarrow \infty$.
		\item For any essential annulus $E\subset W$ disjoint from $P$, there is a component $c$ of $\partial E$ such that $\{m(m_i|S,c,\mu)\}$ is bounded for the component $S$ of $\partial M$ on which $c$ lies.
		\item If $M$ is an $I$-bundle, then $P\neq\emptyset$.
\end{enumerate}
	
	Then the sequence of the restrictions $\{\rho_i| \pi_1(W)\}$ has a convergent subsequence up to conjugation.
\end{theorem}

\begin{proof}
We follow the argument of \cite[Proposition 6.1]{BBCL} with some modifications as below.
Denote by $c_i$  a shortest pants decomposition of $\partial M$  with respect to $m_i$. 
Note that $\{d_Y(m_i, \mu)=d_Y(\mu(m_i), \mu)\}$ is bounded for any essential subsurface $Y$ that is not an annulus with its core curve in $c_i$ if and only if so is $\{d_Y(c_i, \mu)\}$. 
%
Let $r_0$ be a multicurve consisting of core curves of $P$, one taken from each component of $P$.
We define $r_0$ to be empty if $P$ is empty.
Let $X$ be the characteristic submanifold of $(W,P)$. 
Consider a multicurve $r$ of $(W\setminus X)\cap\partial W$ containing $r_0$ which is maximal in the sense that any simple closed curve in $(W\setminus X)\cap\partial W$ either intersects $r_0$ or is homotopic on $\partial W$ to a component of $r_0$. By our assumption (\ref{P}) and  \cref{broken window}, there is $L$ such that $\len_{\rho_i}(r)\leq L$. 

Let $Z$ be the union of the characteristic $I$-pairs in $X$. 
By assumption, $Z$ is a product $I$-bundle in the form $\Sigma\times I$ ($\Sigma$ may  be disconnected). 
We denote by $f \colon Z\to \Sigma$ the projection along the fibres,  and for a subsurface $F\subset \Sigma$ and for $j=0,1$, we use the symbol $F_j$ to denote $f^{-1}(F)\cap \Sigma\times\{j\}$. 
For each component $F$ of $\Sigma$ that is not a pair of pants, by the assumption (c), there is $j\in\{0,1\}$ such that $\{d_{F_j}(c_i,\mu)\}$ is bounded. 
Let $S_j$ be the component of $\partial M$ containing $F_j$, and denote by $\theta_i=\rho_i\circ I_*:\pi_1(S_j)\to\PSL$ the representation induced by the inclusion $I \colon S_j \hookrightarrow M$. 
The quotient manifold $\hthree/\theta_i(\pi_1(S_j))$ covers $\hthree/\rho_i(\pi_1(M))$, and has end invariant $m_i|_{S_j}$ on one side. 
Now, replacing $\rho_i$ with $\theta_i$,  we can follow the proof of \cite[Lemma 6.2]{BBCL} starting at the penultimate paragraph. This gives us a constant $L'$ and a sequence of curve $\{a_i\}$ on $F$ such that $\ell_{\rho_i}(a_i)\leq L'$ and $\{d_{F_j}(f^{-1}(a_i) \cap F_j,c_i)\}$ is bounded.

Up to isotopy, $f(r\cap Z)$ consists of boundary components of $\Sigma$. 
We denote $f(r\cap Z)$ by $s$. 
Still following \cite{BBCL}, if $\{a_i\}$ has a constant subsequence, then we pass to an
appropriate subsequence of $\{\rho_i\}$,  and add $a_i$ (independent of $i$) to $s$. 
If not, by \cite[Lemma 2.3]{BBCL}, there is a subsurface $Y\subset F$ with  $d_Y(a_i,\mu)\longrightarrow \infty$, passing to a subsequence. 
Since $\{d_{F_j}(c_i,\mu)\}$ and $\{d_{F_j}(f^{-1}(a_i) \cap F_j,c_i)\}$ are bounded, $Y$ must be a proper subsurface of $F$ (even up to isotopy). 
If,  passing to a subsequence, there is $k\in\{0,1\}$ such that $Y_k=f^{-1}(Y)\cap S_k$ is an annulus containing a component of $c_i$ for all $i$, we add the projection by $f$ of this component of $c_i$ to $s$. Otherwise, by the assumption (c), there exists
$k\in\{0,1\}$ with bounded  $\{d_{Y_k}(c_i,\mu)\}$. Hence, passing to a subsequence, $d_{Y_k}(c_i,f^{-1}(a_i) \cap S_k)\longrightarrow\infty$, and by \cite[Theorem B]{minsky-gt}, $\ell_{\rho_n}(\partial Y)\to 0$. In this case, we add $\partial Y$ to $s$. We repeat the above construction letting $F$ be a component of $\Sigma \setminus s$ until $\Sigma \setminus s$ becomes a union of annuli and pair of pants.
Adding $f^{-1}(s)$ to $r$, we obtain a pants decomposition of $\partial W$, which we shall still denote by $r$, such that  $\{\ell_{\rho_n}(r)\}$ is bounded.

Next we attach a transversal with bounded length to each component of $s$. 
Let $e$ be a curve in $s$, by the assumption (c), there is $j\in\{0,1\}$ such that $\{m(m_i,e_j,\mu)\}$ is bounded, where $e_j=f^{-1}(e) \cap S_j$. We replace $c_i$ with a shortest pants decomposition not containing $e_j$. 
Since $\{m(m_i,e_j,\mu)\}$ is bounded, there is a positive lower bound on $\{\ell_{m_i}(e_j)$\}, and there is an upper bound on $\{\ell_{m_i}(c_i)\}$ by our definition of $c_i$. 
Considering the covering associated with the inclusion $S_j\hookrightarrow M$ we can use the arguments of \cite{BBCL} (proof of Proposition 6.1, the part after the proof of Lemma 6.2) to obtain a transversal $t_e$ to $e$ with bounded length $\ell_{\rho_i}(t_e)$.

Since the union of $s$ and all  transversals defined for the components of $s$ is doubly incompressible in Thurston's sense \cite[Section 2]{Th3}, we can deduce from Thurston's relative boundedness theorem \cite[Theorem 3.1]{Th3} that the restriction of $\{\rho_i|_{\pi_1(W)}\}$ has a convergent subsequence.
\end{proof}

\begin{generalization}
The following is a refinement of \cref{relative convergence} which is adapted to the setting of \cref{more general}.
The proof is quite similar.

\begin{theorem}
\label{pared relative convergence}
Let $(M,Q)$ be an atoroidal  boundary-irreducible pared $3$-manifold all of whose characteristic $I$-pairs are product $I$-bundles.
 Let $\{m_i\}$ be a sequence in $\teich(\partial M \setminus Q)$, and $\{\rho_i:\pi_1(M)\to \PSL\}$  a sequence of representations corresponding to $\{q(m_i)\}$, where $q \colon \teich(\partial M \setminus Q) \to QH(M, \gamma_Q)$ is the parametrisation given in \S1, and $\gamma_Q$ denotes the union of core curves of $Q$. 
 Let $\mu\subset \partial M \setminus $ be a full and clean marking, and  $W\subset M$ a submanifold with paring locus $P$.
 We assume the following:
\begin{enumerate}[(a)]
\item $P$ contains $Q$.
		\item the closure  of $W\setminus \partial M$ is a union of essential annuli contained in $P$
		\item 
		\label{P}
		for any non-contractible simple closed curve $c$ in $P$,  $\len_{\rho_i}(c)$ is bounded as $i \longrightarrow \infty$;
		\item for any essential annulus $E\subset W$ disjoint from $P$, there is a component $c$ of $\partial E$ such that $\{m(m_i|_S,c,\mu)\}$ is bounded for the component $S$ of $\partial M$ on which $c$ lies.
\end{enumerate}

\end{theorem}

\end{generalization}

\section{Unbounded skinning and annuli}
\label{main prop}
The following proposition is the main step of our proof of \cref{main}.

\begin{proposition}	\label{unbounded has a root}
Let $M$ be an orientable atoroidal boundary-irreducible Haken $3$-manifold  that is strongly untwisted. 
Let $\{m_i\}$ be a sequence in $\teich(\partial M)$, let $\sigma$ be the skinning map, and assume that there is a simple closed curve $d$ on $\partial M$ such that $m(\sigma(m_i),d,\mu)\longrightarrow\infty$ for a full clean marking $\mu$. Then, passing to a subsequence, there is a properly embedded essential annulus $A\subset M$ with  $\partial A=d\cup d'$ such that $m(m_i,d',\mu)\longrightarrow\infty$.
\end{proposition}

\begin{generalization}
In the setting of \cref{more general}, we need to refine \cref{unbounded has a root} as follows.

\begin{proposition}
\label{relative unbounded}
Let $(M,Q)$ be an orientable boundary-irreducible pared manifold all of whose characteristic $I$-pairs are product $I$-bundles.
Let $\{m_i\}$ be a sequence in $\teich(\partial M \setminus Q)$, and $\sigma_h \colon \teich(\partial_0 M \setminus Q) \to \teich(\bar \partial_0 M \setminus Q)$ be the skinning map as in the setting of \cref{more general}.
Suppose that there is an essential simple closed curve $d$  on $\partial_0 M \setminus P$ such that $m(\sigma_h(m_i), d, \mu) \longrightarrow \infty$.
Then passing to a subsequence, there is a properly embedded essential annulus $S \subset M$ with $\partial A=d\cup d'$ such that $d' \subset \partial_0 M \setminus Q$ and $m(m_i, d', \mu) \longrightarrow \infty$.
\end{proposition}
In the following, we shall principally prove \cref{unbounded has a root}, whereas we shall put remarks where we need modifications for the proof of \cref{relative unbounded}.
\end{generalization}
%
%
%
We are going to show that any subsequence of $\{m_i\}$ contains a further subsequence for which the conclusion holds. To simplify the notations we shall use the same subscript $i$ for all subsequences.

\subsection{Re-marking}
Our manifold $M$ is either connected or has two components.
In the case when $M$ has two components, by considering the component on which $d$ lies, and abusing the symbol $M$ to denote this component, we can assume that $M$ is connected.
Recall that, by the assumption throughout this section, $M$ is strongly untwisted.
Let $\rho_i \colon \pi_1(M)\to \pslc$ be a representation corresponding to $q(m_i)$.

%

As a first step for the proof of \cref{unbounded has a root}, we change the markings of $M$ so that the behaviour of the $\rho_i$ can be read more easily from the behaviour of their end invariants.

\begin{lemma}		\label{lemma:diffeo}
Let $d$ be an essential simple closed curve on $\partial M$,  and let $d_1,\dots ,d_p$ be disjoint simple closed curves on $\partial M$ representing the homotopy classes of simple closed curves on $\partial M$ homotopic to $d$ in $M$, where $d_1=d$.
Furthermore, we assume that
\begin{enumerate}
\item[(*)]  $\{m(m_i,d_j,\mu)\}$ is bounded for every $j=2,\dots , p$.
\end{enumerate}
Then there is a sequence of orientation-preserving homeomorphisms $\{\psi_i:M\to M\}$ such that, passing to a subsequence, the following hold:
\begin{enumerate}[(1)]
\item For any essential simple closed curve $c\subset\partial M$, either $\{m(\psi_{i*}(m_i),c,\mu)\}$ is bounded or $m(\psi_{i*}(m_i),c,\mu)\longrightarrow\infty$,\label{property:bd or infty}.
\item If $A\subset M$ is an essential annulus disjoint from all the $d_j$ such that $m(\psi_{i*}(m_i),\partial_k A,\mu)\longrightarrow\infty$ for both boundary components $\partial_1 A$ and $\partial_2 A$ of $A$, then $\ell_{\rho_i\circ \psi^{-1}_{i*}}(\partial A)\longrightarrow 0$. \label{property:goes to 0}
\item For every $d_j$ among  $d_1, \dots , d_p$ defined above, 
\label{property:same behaviour}
\begin{enumerate}[(i)]
	 \item $\psi_i(d_j)=d_j$ for every $i$ and $j$,
	 \item for every $j=1, \dots ,p$,  $\{m(\psi_{i*}(m_i),d_j,\mu)\}$ is bounded if and only if $\{m(m_1,d_j,\mu)\}$ is bounded, and
	 \item $\{m(\sigma \circ \psi_{i*}(m_i),d_j,\mu)\}$ is bounded if and only if $\{m(\sigma(m_i),d_j,\mu)\}$ is bounded.
\end{enumerate}

\end{enumerate}
\end{lemma}


\begin{proof}
We shall first define the homeomorphisms $\psi_i$, and then verify the desired properties.
Let $\Xi$ be a component of the characteristic submanifold $X$ of $M\setminus d$. Suppose first that $\Xi$ is a solid
\begin{generalization}
or thickened
\end{generalization} 
torus.
The components of $\overline{\partial \Xi\setminus \partial M}$ are incompressible annuli.
We define $\psi_i$ on soid-torus components $\Xi$ of $X$ to be  a composition of Dehn twists along these frontier annuli with the following properties:
\begin{enumerate}[(a)]
	\item If $\Xi$ is a solid torus, then $\pi_{F}(\mu(\psi_{i*}(m_i)))$ is constant with respect to $i$ for every component $F$ of $\Xi\cap\partial M$ except for at most one.\label{solid torus}
\begin{generalization}
	\item If $\Xi$ is a thickened torus, then $\pi_{F}(\mu(\psi_{i*}(m_i)))$ is constant with respect to $i$ for every component $F$ of $\Xi\cap\partial M$.\label{thickened torus}
\end{generalization}
\end{enumerate}
By the assumption  (*),  passing to a subsequence, we need not compose Dehn twists along annuli of the frontier components of $\Xi$ to achieve the condition (\ref{solid torus}) when $\Xi \cap \partial M$ contains an annular neighbourhood of  $d$ (up to isotopy), and hence $\psi_i$, as defined for the moment,  also satisfies the following:

\begin{enumerate}[(a)]
\addtocounter{enumi}{1}
\item For every $j=1, \dots, p$, we have $\psi_i(d_j)=d_j$ and $\pi_{A_j}(\mu(\psi_{i*}(m_i)))=\pi_{A_j}(\mu(m_i))$ for an  annulus $A_j$ on $\partial M$ whose core curve is $d_j$.\label{homotopic to d}
\end{enumerate}

If $\Xi$ is not a solid
\begin{generalization}
	or thickened
\end{generalization} 
 torus,  $\Xi$ is a product $\Xi= F\times I$. 
(Recall that we have an assumption that every characteristic $I$-pair of $M$ is a product bundle. This implies that an $I$-pair in the characteristic submanifold $X$ of $M \setminus d$ is also a product $I$-bundle.)
Let  $F_0$ be a component of $\Xi\cap\partial M$ which does not contain a curve homotopic on $\partial M$ to $d_1$ (there is always such a component since $M$ is strongly untwisted). 
Since the curve complex of $F_0$ has finitely many orbits under the action of the mapping class group of $F_0$ (relative to $\partial F_0$), there is a sequence of orientation-preserving homeomorphisms $g_i:F_0\to F_0$ fixing $\partial F_0$ such that, passing to a subsequence, $\pi_{F_0}(\mu(g_{i*}(m_i)))$ is constant. 
We then define $\psi_i$ on $\Xi$ by extending $g_i$ along the fibres, i.e. $\psi_i(x,t)=(g_i(x),t)$ for any $(x,t)\in \Xi= F_0\times I$.

Thus we have the following.
\begin{enumerate}[(a)]
	\addtocounter{enumi}{2}
	\item there are $R>0$ and a component $F_0$ of $\Xi\cap\partial M$ not containing any curve homotopic on $\partial M$ to $d_1$ such that  $d_Y(\mu(\psi_{i*}(m_i)),\mu)\leq R$ for any incompressible subsurface $Y\subset F_0$. 
	\label{almost constant}
	\label{I-bundle}
\end{enumerate}

We note that since $\Xi$ is a component of the characteristic submanifold of $M\setminus d$, if $\partial \Xi$ contains a curve $d_j$, then it must be peripheral, and hence the action of $\psi_i$ on $\Xi$ does not affect the property (\ref{homotopic to d}).

We repeat the construction above for all the components of the characteristic submanifold $X$, and we extend the resulting homeomorphisms to   a homeomorphism of $M$ which is isotopic to the identity on the complement of the characteristic submanifold.

We now verify the properties (\ref{property:bd or infty}, \ref{property:goes to 0}, \ref{property:same behaviour}) for $\psi_i$ thus constructed.

The first property (\ref{property:bd or infty}) can be obtained by passing to a subsequence for any sequence of homeomorphisms. 
Therefore, we are done with (\ref{property:bd or infty}).

We next turn to proving the property (\ref{property:same behaviour}).
By the assumption (*), taking a subsequence, we may assume that $\pi_{F}(\mu(m_i))$ is constant whenever $F$ is an annulus containing a curve $d_j$ for $j\neq 1$.
Wet first show the following claim.
\begin{claim}	\label{psi bounded}
For every $j=1, \dots ,p$ and for any sequence of incompressible subsurfaces $Y_i\subset\partial M$  with its boundary containing $d_j$ which are not a pair of pants, $\{d_{Y_i}(\mu,\psi_i(\mu))\}$ is bounded.
\end{claim}

\begin{proof} 
Fix $j=1, \dots , p$, and consider a sequence of incompressible subsurfaces $Y_i\subset\partial M$ each of which contains $d_j$ in its boundary. 
If all of the $Y_i$ are annulli after passing to a subsequence, the conclusion follows from the property (\ref{homotopic to d}). 
From now on, taking a subsequence, we assume that none of the $Y_i$ are annuli.

Assume first that there is a simple closed curve $c\subset \partial M$ intersecting $Y_i$ which lies outside the characteristic submanifold $X$.
Then by our construction of $\psi_i$, we have $\psi_i(c)=c$, and hence

\begin{align*}
	d_{Y_i}(\mu,\psi_i(\mu))&\leq d_{Y_i}(\mu,c)+d_{Y_i}(c,\psi_i(\mu))\\
	&\leq d_{Y_i}(\mu,c)+d_{Y_i}(\psi_i(c),\psi_i(\mu))\\
	&\leq d_{Y_i}(\mu,c)+ d_{\psi_i^{-1}(Y_i)}(\mu,c)\\
	&\leq 4 i(c,\mu)+2,
\end{align*}
where the last inequality is due to Masur--Minsky \cite[Lemma 2.1]{MM}.
Thus we are done in this case.

Otherwise, taking a subsequence, we may assume that $Y_i$ is contained in $\Xi_i \cap \partial M$ for a component $\Xi_i$ of the characteristic submanifold $X$. 
Taking a further subsequence, we may assume that $\Xi_i=\Xi$ does not depend on $i$. 
Since $Y_i$ is not an annulus, $\Xi$ is a product $I$-pair $F\times I$. 
Let $F_0$ be the component of $\Xi\cap\partial M$ given by the property (\ref{I-bundle}). 
Let us denote by $Y_i'$ the projection of $Y_i$ to $F_0$ along the fibres, (setting $Y_i'=Y_i$ if $Y_i\subset F_0$). 
By our definition of $d_1, \dots, d_p$, the boundary of $Y'_i$ contains some $d_k$ with $k \geq 2$. 
Then,  $\{m(m_i,d_k,\mu)\}$ is bounded by the assumption (*), and $\psi_i(d_k)=d_k$ by the property (\ref{homotopic to d}). 
In particular $\{d_{Y'_i}(\mu(\psi_i(m_i)),\psi_i(\mu))=d_{\psi_i^{-1}(Y_i')}(\mu(m_i),\mu)\}$ is bounded. 
On the other hand, by the property (\ref{I-bundle}), $\{d_{Y_i'}(\mu(\psi_{i*}(m_i)),\mu)\}$ is bounded. 
Thus we see that $\{d_{Y'_i}(\mu,\psi_i(\mu))\leq d_{Y'_i}(\mu,\mu(\psi_{i*}(m_i))+d_{Y'_i}(\mu(\psi_{i*}(m_i),\psi_i(\mu))\}$ is bounded. 
It follows from the construction of $\psi_i$ that $d_{Y_i}(\mu,\psi_i(\mu))=d_{Y'_i}(\mu,\psi_i(\mu))$,  and hence $\{d_{Y_i}(\mu,\psi_i(\mu))\}$ is also bounded.
\end{proof}

Now we can show that the sequence $\{\psi_i\}$ satisfies the property (\ref{property:same behaviour}) by the condition (*) and the following claim.

\begin{claim}	\label{cl: same behaviour}
For any $j=1, \dots , p$, the sequence $\{m(\psi_{i*}(m_i),d_j,\mu)\}$ is bounded if and only if $\{m(m_i, d_j,\mu)\}$ is bounded, and $\{m(\sigma\circ \psi_{i*}(m_i),d_j,\mu)\}$ is bounded if and only if $\{m(\sigma(m_i),d_j,\mu)\}$ is bounded.
\end{claim}

\begin{proof}
%
Let $\{Y_i\subset\partial M\}$ be a sequence of incompressible subsurfaces with $d_j\subset\partial Y_i$ which are not pairs of pants.
Since $d_{Y_i}(m_i,\mu)=d_{\psi_i(Y_i)}(\mu(\psi_{i*}(m_i)),\psi_i(\mu))$, the triangle inequalities \begin{equation*}\begin{split}
&d_{Y_i}(\mu(m_i),\mu)\leq d_{\psi_i(Y_i)}(\mu(\psi_{i*}(m_i)),\mu)+d_{\psi_i(Y_i)}(\mu,\psi_i(\mu)), \text{ and }\\ &d_{\psi_i(Y_i)}(\mu(\psi_{i*}(m_i)),\mu)\leq d_{\psi_i(Y_i)}(\mu(\psi_{i*}(m_i)),\psi_i(\mu))+d_{\psi_i(Y_i)}(\psi_i(\mu),\mu)\end{split}
\end{equation*}
 lead to
\begin{equation*}
\begin{split} &d_{\psi_i(Y_i)}(\mu(\psi_{i*}(m_i)),\mu)-d_{\psi_i(Y_i)}(\mu,\psi_i(\mu))\\
&\leq d_{Y_i}(\mu(m_i),\mu)\leq d_{\psi_i(Y_i)}(\mu(\psi_{i*}(m_i)),\mu)+d_{\psi_i(Y_i)}(\mu,\psi_i(\mu)).\end{split}
\end{equation*}
Thus by applying Claim \ref{psi bounded}, we see that $\{d_{Y_i}(\mu(m_i),\mu)\}$ is bounded if and only if $\{d_{\psi_i(Y_i)}(\mu(\psi_{i*}(m_i)),\mu)\}$ is bounded. 

Since $\psi_i(d_j)=d_j$ by the property (\ref{homotopic to d}), we also have $\len_{m_i}(d_j)=\len_{\psi_{i*}(m_i)}(d_j)$, and we conclude that $\{m(\psi_{i*}(m_i),d_j,\mu)\}$ is bounded if and only if $\{m(m_i,d_j,\mu)\}$ is bounded.
	
Since $\sigma$ commutes with $\psi_{i*}$, the same argument shows that  that $\{m(\sigma\circ \psi_{i*}(m_i),d,\mu)\}$ is bounded if and only if $\{(\sigma(m_i), d, \mu)\}$ is bounded.
\end{proof}

To conclude the proof of Lemma \ref{lemma:diffeo}, it remains to establish the property (\ref{property:goes to 0}).
We restate the property as a claim.

\sloppy
\begin{claim}	\label{lma: goes to 0}
Let $A\subset M$ be an essential annulus with its boundary components denoted by $\partial_1A$ and $\partial_2 A$.
Suppose  that $m(\psi_{i*}(m_i),\partial_k A,\mu)\longrightarrow\infty$ for both $k=1$ and $k=2$.
Then $\len_{\rho_i\circ \psi^{-1}_{i*}}(\partial_1 A)\longrightarrow 0$.
\end{claim}

\begin{proof}
Let $a_1, \dots , a_q$ be homotopically distinct simple closed curves  on $\partial M$ representing all the homotopy classes (in $\partial M$) homotopic to $\partial_1 A$ in $M$.
By renumbering them, we can assume $a_k=\partial_k A$ for $k=1,2$.
If $\len_{\psi_{i*}(m_i)}(a_k)\longrightarrow 0$ for some $k=1, \dots, q$, we are done. 

To deal with the remaining case, we now assume that there is a positive constant $\eps$ such that $\len_{\psi_{i*}(m_i)}(a_k)\geq\eps$ for every $i\in\naturals$ and $k=1, \dots , q$. 
Then, there are a constant $L$ and simple closed curves $c_{k,i}$ for every $i\in\naturals$ and $k=1, \dots, q$ such that $c_{k,i}$ intersects $a_k$ essentially and $\len_{\psi_{i*}(m_i)}(c_{k,i})\leq L$. 
There is also $K_1$ such that $d_Y(c_{k,i},\mu(\psi_{i*}(m_i)))\leq K_1$ for any $j,i$ and any incompressible subsurface $Y\subset\partial M$ intersecting $c_{k,i}$ that is neither an annulus nor a pair of pants, since by definition, the length of $\mu(\psi_{i*}(m_i))$  is also bounded from above by a constant.

Since $m(\psi_{i*}(m_i),a_k,\mu)\longrightarrow\infty$ and $\ell_{\psi_{i*}(m_i)}(a_k)\geq\eps$ for $k=1,2$,  there are incompressible subsurfaces $Y_{k,i}$ such that $a_k\subset\partial Y_{k,i}$ and $d_{Y_{k,i}}(\mu(\psi_{i*}(m_i)),\mu)\longrightarrow\infty$ for $k=1,2$. If, passing to a subsequence, $Y_{1,i}$ and $Y_{2,i}$ are both annuli, then, up to homotopy, they lie on the boundary of the same component $\Xi$ of the characteristic submanifold (which is, up to passing to a further subsequence independent of $i$). However, the assumption that $m(\psi_{i*}(m_i),a_k,\mu)\longrightarrow\infty$ contradicts (\ref{solid torus}) when $\Xi$ is a solid torus,
\begin{generalization}
(\ref{thickened torus}) when $\Xi$ is a thickened torus
\end{generalization}
and (\ref{I-bundle}) when $\Xi$ is an $I$-pair. Therefore, we can assume that one of the $Y_{k,i} (k=1,2)$, say $Y_{1,i}$ is not an annulus.

Suppose now that  $Y_{1,i}$ is not eventually contained in the characteristic submanifold $X$ (up to homotopy), even after passing to a subsequence.
By taking a subsequence, we can assume that none of the $Y_{1,i}$ are contained in $X$.
Then, there is a simple closed curve $c\subset \partial M$ disjoint from $X$ which intersects $Y_{1,i}$ for all $i$, by passing to a further subsequence. 
By \cref{broken window} there is a constant $L$ such that $\len_{\rho_i \circ \psi_{i*}^{-1}}(c)\leq L$. Since $d_{Y_{1,i}}(\mu(\psi_{i*}(m_i)),\mu)\longrightarrow\infty$ by our assumption, we have $d_{Y_{1,i}}(c_{1,i},c)\longrightarrow\infty$. 
Then it follows from \cite[Theorem B]{minsky-gt} that $\len_{\rho_i\circ \psi_{i*}^{-1}}(\partial Y_{1,i})\longrightarrow 0$, and hence in particular, we have $\len_{\rho_i\circ \psi_{i*}^{-1}}(\partial_1 A)\longrightarrow 0$.

Next suppose that $Y_{1,i}$ eventually lies in $X$. Taking a subsequence, we can assume that all the surfaces $Y_{1,i}$ lie in the same component $\Xi$ of $X$. Since $Y_{1,i}$ is not an annulus, $\Xi$ must be an $I$-bundle, which has a form of $\Xi= F \times I$.
By (\ref{I-bundle}), there is another surface $Y_{3,i}\subset\partial \Xi$ such that $Y_{1,i}$ and $Y_{3,i}$ bound an $I$-bundle compatible with the $I$-bundle structure of $\Xi$, and are projected along the fibres of $\Xi= F\times I$ to the same surface $Z_i$ in $F$ and $d_{Y_{3,i}}(\mu(\psi_{i*}(m_i)),\mu)\leq R$. 
We note that by our definition of $a_1, \dots, a_q$, there is $k_0 \geq 2$ such that $a_{k_0}$ lies on $\partial Y_{3,i}$.
Then since $d_{Y_{3,i}}(\mu(\psi_{i*}(m_i)),c_{k_0,i}) \le K_1$, we have $d_{Y_{3,i}}(c_{k_0,i}, \mu) \le R+K_1$.
We shall make use of $\{c_{1,i}\}$ and $\{c_{k_0,i}\}$ to apply  \cite[Theorem B]{minsky-gt} as before. 
Since they do not lie on the same surface, we first need to project them to $F$. This leads to the following claim:

\begin{claim}
There are $K>0$ and two sequences of simple closed curves $\{d_{1,i}\}$ and $\{d_{k_0,i}\}$ on $F$ such that $\len_{\rho_i \circ \psi_{i*}^{-1}}(d_{k,i})\leq K$ for all $i$ and $k=1,k_0$, and $d_{Z_i}(d_{1,i}, d_{k_0,i})\longrightarrow\infty$.
\end{claim}
\begin{proof}
Let $k$ be either $1$ or $k_0$.
If $c_{k,i}$ is contained in $\Xi$ for sufficiently large $i$, then we let $d_{k,i}$ be the projection of $c_{k,i}$ to $F$. 
We also note that $\len_{\rho_i \circ \psi_{i*}^{-1}}(d_{k,i}) \leq L$ then.

Suppose that this is not the case.
We let $S$ be the component of $\partial M$ containing $c_{k,i}$. 
Following \cite[page 138]{minsky-gt} we extend the multicurve $B:=\Fr (\Xi\cap S)$ to a complete geodesic lamination $\lambda$ by performing Dehn twists   around $B$ infinitely many times to $c_{k,i}$ and adding finitely many isolated leaves spiralling around $B$. 
There is a unique pleated surface $h_{k,i}\colon S\to\hthree/\rho_i(\pi_1(S))$ realising $\lambda$ which induces $\rho_i \circ \psi_{i*}^{-1}$ between the fundamental groups. 
Let $R_\lambda$ be the  $\eps$-thick part of $S$ with respect to the hyperbolic metric induced by $h_{k,i}$.
By the efficiency of pleated surfaces (\cite[Theorem 3.3]{Th2}, \cite[Theorem 3.5]{minsky-gt}), there is a constant $K_2$ such that $\len_{h_{k,i}}(c_{k,i}\cap R_\lambda)\leq L+K_2 i(c_{k,i},B)$ (the relation between the alternation and intersection numbers comes from (4.3) in \cite{minsky-gt}). 
 In particular, there is an arc $\kappa_{k,i}$ in $c_{k,i}\cap (\Xi\cap S)\cap R_\lambda$ intersecting $Y_{k,i}$ and having length at most $L+K_2$. 
By \cref{broken window}, the length of each component of $B$ on $h_{k,i}$ is bounded by a constant $L'$ independent of $i$.
By joining one or two copies of $\kappa_{k,i}$ (depending on whether $\kappa_{k,i}$ intersects one or two components of $B \cup \Fr R_\lambda$) with arcs on $B \cup \Fr R_\lambda$, we can construct in $S\cap \Xi$ a simple closed curve $d_{k,i}$ such that $\len_{h_{k,i}}(d_{k,i})\leq 2(L+K_2+L'+\eps)$. 
Furthermore,  this construction implies that there is a constant $K_3$ such that $d_Y(d_{k,i},c_{k,i})\leq K_3$ for any incompressible subsurface $Y\subset S\cap \Xi$ intersecting both $d_{k,i}$ and $c_{k,i}$, and in particular for $Y=Y_{k,i}$. We use the same symbol $d_{k,i}$ to denote the projection of $d_{k,i}$ on $F$ along the fibres of $\Xi=F\times I$.

Thus we have $\len_{\rho_i\circ \psi_{i*}^{-1}}(d_{k,i})\leq 2(L+K_2+L'+\eps), $ and $d_{Z_i}(d_{1,i}, d_{k_0,i})\geq d_{Y_{1,j}}(c_{1,i},\mu)-d_{Y_{k_0,i}}(c_{k_0,i},\mu)-2K_3 \geq d_{Y_{1,j}}(c_{1,i},\mu)-R-K_1-2K_3 \longrightarrow\infty$.
\end{proof}

\noindent
{\em Continuation of Proof of \cref{lma: goes to 0}.}
Set $\vartheta_i=\rho_i\circ \psi_{i*}^{-1} \circ I_*:\pi_1(S)\to\PSL$ where $I_*:\pi_1(S)\to\pi_1(M)$ is the homomorphism induced by the inclusion. 
Following \cite{minsky-gt}, we denote by $\mathcal C_0(\vartheta_i, K)$ the set of simple closed curves on $S$ whose translation lengths with respect to $\vartheta_i$ are less than or equal to $K$.
By the claim above, we see that both $d_{1,i}$ and $d_{k_0,i}$ lie in $\mathcal{C}_0(\vartheta_i, K)$ and that  $d_{Y_{1,i}}(d_{1,i},d_{k_0,i})\longrightarrow\infty$. In particular, $\mathrm{diam}_{Y_{1,i}}({\mathcal C}_0(\vartheta_i, K))\longrightarrow\infty$. It follows from \cite[Theorem B]{minsky-gt} that $\len_{\vartheta_i}(\partial Y_{1,i})\longrightarrow 0$. In particular, $\len_{\vartheta_i}(\partial_1 A)\longrightarrow 0$, and hence  $\len_{\rho_i\circ \psi_{i*}^{-1}}(\partial_1 A)\longrightarrow 0$. 
\end{proof}

This also concludes the proof of Lemma \ref{lemma:diffeo}.
\end{proof}

By Claim \ref{cl: same behaviour}, proving Proposition \ref{unbounded has a root} for $\{\rho_i\}$ is equivalent to proving it for $\{\rho_i\circ \psi_{i*}^{-1}\}$. Thus we may assume that $\{\rho_i\}$ satisfies the following.
{\em
\begin{enumerate}[{\rm (I)}]
	\item For any simple closed curve $c\subset\partial M$, either $\{m(m_i,c,\mu)\}$ (resp.\ $\{m(\sigma(m_i),c,\mu)\}$,) is bounded or $m(m_i,c,\mu)\longrightarrow\infty$  (resp. $m(\sigma(m_i),c,\mu)\longrightarrow \infty)$.
	\item If $A\subset M$ is an essential annulus such that $m(m_i,\partial_k A,\mu)\longrightarrow\infty\ (k=1,2)$ for both boundary components $\partial_1 A$ and $\partial_2 A$ of $A$, then $\len_{\rho_i}(\partial A_1^*)\longrightarrow 0$.
\end{enumerate}
}

\subsection{End invariants and wrapping}
\label{control}

In this subsection, we shall discuss how algebraic limits projects to geometric limits and how this is reflected in the behaviour of the end invariants.

Let us now fix the assumptions and notations which will be used in most results of this section.

\begin{setting}	\label{long setting}
We consider an orientable atoroidal compact boundary-irreducible Haken $3$-manifold $M$ without  torus boundary components,  and a sequence of representations $\rho_i\in QH(M)$ corresponding to  Ahlfors-Bers coordinates $m_i\in\teich(\partial M)$.
We have a non-contractible simple closed curve $d\subset\partial M$, and we denote by $d_1,...,d_p\subset \partial M$  simple closed curves representing  all homotopy classes of $\partial M$  on $\partial M$ which are homotopic to $d$  in $M$, with $d=d_1$. We assume that $\ell_{\rho_i}(d^*)\longrightarrow 0$.

%
%
We also assume that we have a submanifold $V_d$ of $M$ 
whose frontier consists of incompressible annuli and which has the following three  properties:

\begin{enumerate}[(i)]
\item $V_d$ contains all the curves $d_j\ (j=1,\dots , p)$, and $d_j$ is not peripheral in $V_d\cap\partial M$ for every $j=1, \dots ,p$. 
\label{d contained}
\item The restriction of $\rho_i$ to $\pi_1(V_d)$ converges to a representation $\rho_\infty\colon \pi_1(V_d) \to \pslc$. 
\item If $A\subset V_d$ is an essential annulus disjoint from $d$ with core curve $a$ which is not homotopic to $d$ in $M$, then  $\len_{\rho_i}(a)\longrightarrow 0$ if and only if $A$ is properly homotopic to the closure of a component of $\partial V_d\setminus \partial M$.
\end{enumerate}

Suppose first $p \geq 2$.
 If a component of the characteristic submanifold containing $d$ (up to isotopy) is a solid torus, then it contains all of $d_1, \dots , d_p$ up to isotopy.
We let $T$ be this characteristic solid torus in this case.
If the component is an $I$-pair, then $p=2$, and it contains $d_2$ up to isotopy.
In this case, we let $T$ be $A \times [0,1]$ such that $A\times \{0\}$ is an annular neighbourhood of $d$ whereas $A \times \{1\}$ is that of $d_2$.
Since $\Fr V_d$ consists of annuli, by the condition (\ref{d contained}) above, $T$ can be assumed to be contained in $V_d$ by moving it by an isotopy in both cases.
If $p=1$, we set $T=\emptyset$.

Given $j=1, \dots , p$, we denote by $F_j$ the component of $V_d\cap\partial M\setminus \bigcup_{k\neq j} d_k$ containing $d_j$.

The sequence of groups $\{\rho_i(\pi_1(V_d))\}$ converges geometrically to a Kleinian group $\Gamma$ containing $\rho_\infty(\pi_1(V_d))$, passing to a subsequence.
\end{setting}

In the next section, we shall construct $V_d$ having the properties above, which shows that our argument in the present section really works.

Assuming the existence of $V_d$ for the moment, we now prove that every component of $V_d\setminus T$ has a compact core which  is embedded in the geometric limit $\hyperbolic^3/\Gamma$ making use of the work of \cite{BBCM}.

\begin{lemma}	\label{embedded core for W}
In  \cref{long setting}, let $W$ be a submanifold of $V_d$ which is the closure of a component of $V_d\setminus T$. Then there is a relative compact core $C_W\subset \hthree/\rho_\infty(\pi_1(W))$ which is homeomorphic to $W$ and on which the restriction of the covering projection $\hthree/\rho_\infty(\pi_1(W))\to\hthree/\Gamma$ induced by the inclusion  is injective.
Furthermore, for the closures of two components $W_1, W_2$ of $V_d \setminus T$ (in the case when $T$ is non-empty and separates $W$), the compact cores $C_{W_1}$ and $C_{W_2}$ can be taken so that their images in $\hthree/\Gamma$ are disjoint. 
\end{lemma}

\begin{proof}
Our conditions in \cref{long setting} imply the assumptions of \cite[Proposition 4.4]{BBCM}, and applying this proposition, we see that there is a compact submanifold of $\hthree/\Gamma$ which lifts to a compact core $C_W$ of $\hthree/\rho_\infty(\pi_1(W))$ such that the restriction of the covering projection $\hthree/\rho_\infty(\pi_1(W))\to\hthree/\Gamma$ to $C_W$ is injective. Let $\Gamma_W\subset \Gamma$ be the geometric limit of $\{\rho_i(\pi_1(W))\}$.
Then the restriction of the covering projection $\hthree/\rho_\infty(\pi_1(W))\to\hthree/\Gamma_W$ to $C_W$ must also be injective.
	
By \cite[Lemma 4.6]{BBCM}, $\rho_\infty(\pi_1(W))$ is either a generalised web group or a degenerate group without accidental parabolic elements. It follows then from \cite[Corollary C and Theorem E]{AC} that $C_W$ is homeomorphic to $W$.
The last sentence of our lemma also follows from \cite[Proposition 4.4]{BBCM}.
\end{proof}

We next show that by performing Dehn twists along  embedded annuli bounded by $d$ and  $d_j\ (j=2, \dots , p)$, we can make each $F_j$ embedded in the algebraic limit and mapped injectively in the geometric limit by the covering projection.

In the next lemma and the following, we shall use the expression \lq the {\em outward side} of a cusp'.
We say that an embedding of the surface $F_j\subset\partial V_d$ into the geometric limit $\hthree/\Gamma$ lies on the outward side of a cusp if the cusp lies on the same side of the embedding of $F_j$ as the embeddings of the components of $V_d \setminus T$ intersecting $F_j$.
Otherwise we say that the embedding of $F_j$ lies on the {\em inward side} of the cusp.

\begin{lemma}	\label{twists}
In \cref{long setting}, we denote by $D_j$ the right-hand Dehn twist along an embedded annulus bounded by $d=d_1$ and $d_j$ ($j=2, \dots , p$). 
Then for each $j$, there is a sequence $\{a_i(j)\}$ of integers with the following properties:
\begin{enumerate}[--]
	\item The sequence $\{\theta_i=\rho_i\circ {D_{j*}^{a_i(j)}}|_{\pi_1(F_j)}\}$ converges algebraically to a representation $\theta_\infty :\pi_1(F_j)\to \PSL$.	
	\item There is an embedding $h_j :F_j\to\hthree/\theta_\infty(\pi_1(F_j))$ inducing $\theta_\infty$ such that the restriction of the covering projection $\Pi_{F_j}:\hthree/\theta_\infty(\pi_1(F_j))\to\hthree/\Gamma$ to $h_j(F_j)$ is an embedding and its image $\Pi_{F_j}\circ h_j(F_j)$ lies on the outward side of the cusp corresponding to $\rho_\infty(d)=\theta_\infty(d)$ when the latter is a rank-2 cusp.
\end{enumerate}
\end{lemma}

\begin{proof}
This is a relative version of \cite[Lemma 4.5]{BBCL}.

Let $W'$ and $W''$ be the components of $V_d\setminus T$ intersecting $F_j$ (we set $W'=W''$ if there is only one such component), and set $F_j'=F_j\cap W'$ and $F_j''=F_j\cap W''$. By Lemma \ref{embedded core for W}, there are compact cores $C_{W'}\subset \hthree/\rho_\infty(\pi_1(W'))$ and $C_{W''}\subset \hthree/\rho_\infty(\pi_1(W''))$, homeomorphic to $W'$ and $W''$ respectively, on which the restrictions of the covering projections to $\hthree/\Gamma$ are injective. 
The inclusions induce embeddings $f':F_j'\hookrightarrow\partial C_{W'}$ and $f'':F_j''\hookrightarrow\partial C_{W''}$ which lift to embeddings $g':F_j'\hookrightarrow \hthree/\rho_\infty(\pi_1(F_j))$ and $g'':F_j''\hookrightarrow \hthree/\rho_\infty(\pi_1(F_j))$. 
The restrictions of the covering projection $\Pi_{F_j}:\hthree/\rho_\infty(\pi_1(F_j))\to \hthree/\Gamma$ to $g'(F_j\cap W')$ and to $g''(F_j\cap W'')$ are embeddings. 

If $ T$ does not separate $F_j$, we set $\check g=g'=g''$, otherwise, we put $g'$ and $g''$ together to get an embedding $\check g:F_j\setminus T\to \hthree/\rho_\infty(\pi_1(F_j))$.  
Moving $C_{W'}, C_{W''}$, $f'$ and $f''$ by isotopies, we may assume that they send the boundary of $F_j\setminus T$ into the $\eps$-thin part. Then for an appropriate choice of $\eps$, the map $\check g$ sends the boundary of $F_j\setminus T$ to the boundary of the $\eps_1$-thin part of $ \hthree/\rho_\infty(\pi_1(F_j))$, where $\eps_1$ is smaller than the three-dimensional Margulis constant. 
It is then easy to extend $\check g$ to an embedding $g:F_j\to \hthree/\rho_\infty(\pi_1(F_j))$ such that $g(T\cap F_j)$ lies on the boundary of the $\eps_2$-thin part with $\eps_2\leq\eps_1$. By \cref{embedded core for W} and by our construction, the restriction of $\Pi_{F_j}\circ g$ to $F_j\setminus T$, which is $\Pi_{F_j} \circ \hat g$, is an embedding and with an appropriate choice of $\eps$, the composition $\Pi_{F_j}\circ g$ maps $F_j\cap T$ to the boundary of the $\eps_0$-thin part of $\hthree/\Gamma$.	

	If $\rho_\infty(d)$ belongs to a rank-$1$ maximal parabolic subgroup of $\Gamma$, then it is easy to change $g$ on $F_j\cap T$ so that $\Pi_{F_j}\circ g$ is an embedding. In this case, we simply take $a_i$ to be $0$.

Otherwise, $\rho_\infty(d)$ belongs to a rank-$2$ maximal parabolic subgroup of $\Gamma$. 
We denote by $T_0$ the boundary of the corresponding torus cusp-neighbourhood in $\hthree/\Gamma$, i.e. the boundary of the corresponding component of the $\eps_2$-thin part. 
Let $Z$ be the union of $\Pi_{F_j}\circ g(F_j\setminus T)$ and $T_0$. Then $\Pi_{F_j}\circ g(F_j)$ is contained in $Z$ by our way of extending $\check g$ to $g$ as described above. 
As is explained in \cite[Lemma 3.1]{BBCL}, $\Pi_{F_j}\circ g$ is homotopic to a standard map $f_k$ wrapping  $k$ times  around $T_0$ for some $k\in\integers$, and there are two standard embeddings $f_0,f_1:F_j\to Z$ such that $f_0(F_j)$ lies on the outward side of the cusp associated with $d$  and $f_1(F_j)$ lies on its inward side, both without wrapping around $T_0$.
	
	Let $\{q_i \colon B_{r_i}(\hthree/\rho_i(\pi_1(M)), x_i) \to B_{K_ir_i}(\hthree/\Gamma, x_\infty) \}$ be a sequence of $K_i$-bi-Lipschitz approximate isometry on the $r_i$-ball with $r_i \longrightarrow \infty, K_i\longrightarrow 1$ given by the geometric convergence as explained in \cref{geom limit}. By \cite[Lemma 3.1]{BBCL}, there is $s_i\in\integers$ such that $q_i^{-1}\circ f_0$ is homotopic to $q_i^{-1}\circ\Pi_{F_i} \circ  g\circ D_j^{s_i}$. The conclusion follows, taking $a_i(j)=s_i$ and setting $h_j$ to be the lift of $f_0$ to $\hyperbolic^3/\theta_\infty(\pi_1(F_j))$.
\end{proof}

Next we study how the embedding of a compact core in the geometric limit as above affects the end invariants. 

\begin{lemma}	\label{embedded core => bounded}
In  \cref{long setting}, for each $j=1, \dots , p$, suppose that there is an embedding $h_j:F_j\to\hthree/\rho_\infty(\pi_1(F_j))$ inducing $\rho_{\infty}|_{\pi_1(F_j)}$ such that the restriction of the covering projection $\Pi_{F_j}:\hthree/\rho_\infty(\pi_1(F_j))\to\hthree/\Gamma$ to $h_j(F_j)$ is an embedding.

If $\Pi_{F_j}(h_j(F_j))$ lies on the outward side of the cusp associated with $\rho_\infty(d)\in \rho_\infty(\pi_1(M))\subset\Gamma$, then $\{m(m_i,d_j,\mu)\}$ is bounded whereas $m(\sigma(m_i),d_j,\mu)\longrightarrow\infty$.
If $\Pi_{F_j}(h_j(F_j))$ lies on the inward side of the cusp associated with $\rho_\infty(d)$ then $\{m(\sigma(m_i),d_j,\mu)\}$ is bounded whereas $m(m_i,d_j,\mu)\longrightarrow\infty$.
\end{lemma}

\begin{proof}
Suppose that $\Pi_{F_j}(h_j(F_j))$ lies on the outward side of the cusp associated with $\rho_\infty(d)\in \rho_\infty(\pi_1(M))\subset\Gamma$.
Let $c\subset F_j$ be a simple closed curve intersecting $d_j$ essentially, $c^*$ the closed geodesic homotopic to $\Pi_{F_j}(h_j(c))$, and denote by $\psi_i\colon B_{K_i r_i}(\hthree/\rho_i(\pi_1(M)), x_i) \to B_{r_i}(\hthree/\Gamma, x_\infty)$ an approximate isometry  associated with the geometric convergence of $\{\rho_i(\pi_1(M))\}$ to $\Gamma$ as explained in \cref{geom limit}. 
For $i$ large enough, $\psi_i^{-1}(c^*)$ is a quasi-geodesic lying  outside  the thin part. on the same side as $F_j$ of the Margulis tube associated with $\rho_i(d)$. 
Let $S_j$ be the component of $\partial M$ containing $F_j$. 
In the covering $\hthree/\rho_i(\pi_1(S_j))$ of $\hthree/\rho_i(\pi_1(M))$, the closed geodesic homotopic to $\rho_i(c)$ lies above the Margulis tube associated with $\rho_i(d)$. 
Therefore, by \cite[Theorem 1.3]{BBCM-gt} there is a constant $D$ such that $d_Y(c,\mu(m_i))\leq D$ for any surface $Y\subset S_j$ with $d_j\subset \Fr Y$. 
Thus for any full marking $\mu$, there is $D'$ such that $d_Y(\mu,\mu(m_i))\leq D'$ for any surface $Y\subset S$ with $d_j\subset \Fr Y$.

To conclude that $\{m(m_i,d_j,\mu)\}$ is bounded, it remains to show that $\len_{m_i}(d_j)$ is bounded away from $0$. 
Assume the contrary, that $\len_{m_i}(d_j)\longrightarrow 0$ after passing to a subsequence.
Then, there is an annulus joining the closed geodesic $d^*_j\subset\hthree/\rho_i(\pi_1(S))$ representing $\rho_i(d)$ with $d_j^+\subset\partial C(\hthree/\rho_i(\pi_1(S)))$  corresponding to $d_j$, which lies entirely in the $\eps_i$-thin part with $\eps_i\longrightarrow 0$. 
Since $ \psi_i^{-1}(c^*)$ has bounded length, it cannot intersect such an annulus, whereas $\psi_i^{-1}(c^*)$ lies in a uniformly bounded neighbourhood of the convex core for large $i$. Since $c^*$ and $\Pi_{F_j}(h_j(F_j))$ lie on the same side of the cusp associated with $\rho_\infty(d)$, this contradicts the assumption that $\Pi_{F_j}(h_j(F_j))$ lies on the outward side of the cusp associated with $\rho_\infty(d)\in\Gamma$.

Since $\len_{\rho_i}(d)\longrightarrow 0$ and $\{m(m_i, d_j, \mu)\}$ is bounded, it follows from \cite[Short Curve Theorem]{Mi} that $m(\sigma(m_i),d_j,\mu)\longrightarrow\infty$.
%

A quite similar argument works also when $\Pi_{F_j}(h_j(F_j))$ lies on the inward side of the cusp associated with $\rho_\infty(d)\in\Gamma$.
\end{proof}

\begin{corollary}	\label{embedded core surface}
In \cref{long setting}, assume that $p \geq 2$, and consider $j\leq p$ such that $\{m(m_i,d_j,\mu)\}$ is bounded.
Then there is an embedding $h:F_j\to\hthree/\rho_\infty(\pi_1(F_j))$ inducing $\rho_\infty$ such that the restriction of the covering projection $\Pi_{F_j}:\hthree/\rho_\infty(\pi_1(F_j))\to\hthree/\Gamma$ to $h(F_j)$ is an embedding whose image lies on the outward side of the cusp corresponding to $\rho_\infty(d)$.
\end{corollary}

\begin{proof}
As can be seen in the proof of Lemma \ref{twists}, if $\rho_\infty(d)$ belongs to a rank-$1$ maximal parabolic subgroup of $\Gamma$, then $a_i(j)=0$ for any $i$ and $\theta_\infty=\rho_\infty$.
Therefore, our claim of this corollary follows immediately from \cref{twists,embedded core => bounded}.

Otherwise, $\rho_\infty(d)$ belongs to a rank-$2$ maximal parabolic subgroup of $\Gamma$. By \cref{twists}, 
there is a sequence of integers $\{a_i(j)\}$ and an embedding $h_j \colon F_j \to \hthree/\theta_\infty(\pi_1(F_j))$ inducing $\theta_\infty$ between the fundamental groups such that the restriction of the covering projection $\Pi_{F_j}:\hthree/\theta_\infty(\pi_1(F_j))\to\hthree/\Gamma$ to $h_j(F_j)$ is an embedding and its image $\Pi_{F_j}\circ h_j(F_j)$ lies on the outward side of the cusp corresponding to $\theta_\infty(d)=\rho_\infty(d)$.
By \cref{embedded core => bounded}, $\{m(D_{j*}^{a_i}m_i,d_j,\mu)\}$ is bounded. 
Since $\{m(m_i,d_j,\mu)\}$ is bounded by assumption, this is possible only when $\{a_i(j)\}$ is bounded. Then we may take $a_i(j)=0$ for any $i$ in \cref{twists} so that $\theta_\infty=\rho_\infty$, and the conclusion follows.
\end{proof}

We now put these results together to get the result which we shall use to prove Proposition \ref{unbounded has a root}.

\begin{lemma}	\label{cusp on one side}
In  \cref{long setting}, suppose that $\{m(m_i,d_j,\mu)\}$ is bounded for $j\neq 1$.
Then, 
%
there is a relative compact core for $\hthree/\rho_\infty(\pi_1(V_d))$ homeomorphic to $V_d$ on which the restriction of the covering projection $\Pi_d:\hthree/\rho_\infty(\pi_1(V_d))\to\hthree/\Gamma$ is injective. Furthermore, a cusp neighbourhood corresponding to $\rho_\infty(d)$ intersects the compact core in an annular neighbourhood of $d_1$. 
\end{lemma}

\begin{proof}
By Lemma \ref{embedded core for W}, for the components $W$ of $V_d\setminus T$, we have embeddings $g_W:W\to \hthree/\rho_\infty(\pi_1(V_d))$ inducing $\rho_{\infty}|_{\pi_1(W)}$, on the union of which the restriction of $\Pi_d$ is injective.

If $p=1$, then $T=\emptyset$ by definition, and hence $W=V_d$.
We can take a cusp neighbourhood corresponding to $\rho_\infty(d)$ intersecting $g_W(W)$ along an annulus in the homotopy class of $d$. Since $p=1$, such an annulus is isotopic on $\partial V_d$ to an annular neighbourhood of $d_1=d$.

Suppose that $p\geq 2$, and assume that $\{m(m_i,d_j,\mu)\}$ is bounded for every $j \neq 1$. 
Then by \cref{embedded core surface}, for every $j\neq 1$, there is an embedding $g_j:F_j\to \hthree/\rho_\infty(\pi_1(V_d))$ inducing $\rho_{\infty}|_{\pi_1(F_j)}$ on which the restriction of $\Pi_d$ is injective. Furthermore, it follows from the construction that $g_j$ and $g_W$ agree on $F_j\cap W$. Putting together the maps $g_W$ for all the components $W$ of $V_d \setminus T$ and the $g_j$ for all $j\neq 1$, we get an embedding $g:V_d\to  \hthree/\rho_\infty(\pi_1(V_d))$ inducing $\rho_{\infty}|_{\pi_1(V_d)}$ on which the restriction of $\Pi_d$ is injective.

Changing $g$ by an isotopy, we may assume that $g(V_d)$ intersects a cusp neighbourhood $C$ associated with $\rho_\infty(d)$ along an annulus $A\subset g(\partial V_d)$ which is a regular neighbourhood of $g(d_k)$ for some $k=1, \dots , p$. Then $g(F_k)$ lies on the inward side of $C$. This is possible only if $\Pi_d\circ g(F_k)$ lies on the inward side of $C$; for the restriction of $\Pi_d$ is injective on $g(V_d)$, and hence it cannot wrap around $C$.

By assumption, for every $j\neq 1$, $\{m(m_i,d_j,\mu)\}$ is bounded. It follows then from \cref{embedded core surface} that $\Pi_d(g(F_j))$ lies on the outward side of $C$ for $j\neq 1$. Hence the only possibility is that $A$ is a regular neighbourhood of $g(d_1)$.
\end{proof}

\subsection{Completion of the proof of \cref{unbounded has a root}}

\begin{proof}[Proof of Proposition \ref{unbounded has a root}]
If $M$ is an $I$-bundle, then $m(\sigma(m_i),d, \mu)=m(m_i,d_2, \mu)$ and the conclusion follows. In the other cases, we shall prove the proposition by contradiction. 
Assume that $M$ is not an $I$-bundle, that $m(\sigma(m_i),d, \mu)\longrightarrow\infty$, and that  $\{m(m_i,d_j,\mu)\}$ is bounded for every $j =2, \dots , p$.

By Lemma \ref{lemma:diffeo}, after re-marking and passing to  a subsequence, we may assume that $\{\rho_i=q(m_i)\}$ satisfies:
{\em
	\begin{enumerate}[(1) ]
		\item for any simple closed curve $c\subset\partial M$, either $\{m(m_i,c,\mu)\}$ (resp. $\{m(\sigma(m_i),c,\mu)\}$) is bounded or $m(m_i,c,\mu)\longrightarrow\infty$ (resp. $m(\sigma(m_i),c,\mu)\longrightarrow \infty$);
		\item if $A\subset M$ is an essential annulus such that $\partial A$ does not intersect $d$ (hence any of $d_j$) and $m(m_i,\partial_j A,\mu)\longrightarrow\infty$ for both boundary components $\partial_1 A$ and $\partial_2 A$ of $A$, then $\len_{\rho_i}(\partial A_1^*)\longrightarrow 0$.
	\end{enumerate}
}
We note that by \cref{lemma:diffeo}, $\{m(m_i, d_j, \mu)\}$ is bounded for every $j=2, \dots , p$ and $m(\sigma(m_i),d, \mu)\longrightarrow\infty$ even after re-marking.

Taking a further subsequence, we can also assume that for  any essential annulus $E$ of $M$, either $\len_{\rho_i}(\partial E)\longrightarrow 0$ or $\len_{\rho_i}(\partial E)$ is bounded away from $0$.
Let $\mathcal{A}=\bigcup_k A_k$ be a maximal family of pairwise disjoint non-isotopic essential annuli such that
\begin{enumerate}[{[}i{]}]
\item
 the length of the core curve of each annulus $A_k$ tends to $0$ ($\len_{\rho_i}(\partial A_k)\longrightarrow 0$ for any $k$), 
 \item $\partial \mathcal A$ does not intersect $d$, and 
 \item no component of $\mathcal A$ contains a curve homotopic to $d$. 
 \end{enumerate}
 Denote by $V_d$ the component of $M\setminus N(\mathcal{A})$ containing $d$, where $N(\mathcal{A})$ denotes a thin regular neighbourhood of $\mathcal{A}$. 
 Let $P$ be  the closure of $V_d \setminus \partial M$, which is a union of annuli. Next we shall control the geometry of $V_d$ and the length of $d$.

\begin{claim}	\label{claim:short d}
Passing to a subsequence, the restrictions $\{\rho_{i|\pi_1(V_d)}\}$ converge and $\len_{\rho_i}(d)\longrightarrow 0$.
\end{claim}

\begin{proof}
 Let us first assume that $m(m_i,d_1,\mu)\longrightarrow\infty$, and verify the hypotheses of \cref{relative convergence} with $W=V_d$. The hypothesis (a) follows from the construction of $V_d$. The hypothesis (b) follows from the property (2) above. By \cref{lemma:multicurve m unbounded}, $\{m(m_i,c,\mu)\}$ is bounded for any simple closed curve $c$ intersecting $d$. This observation combined with the assumption that $\{m(m_i,d_j,\mu)\}$ is bounded for any $j=2, \dots p$, the property (2) above and the maximality of $\mathcal A$ yields the hypothesis (c). Now by \cref{relative convergence} we can take a subsequence in such a way that the restrictions $\{\rho_{i|\pi_1(V_d)}\}$ converge.

If $\len_{m_i}(d)\longrightarrow 0$, we are done. Otherwise, since we are assuming that $m(m_i, d, \mu)\longrightarrow \infty$, there is a sequence of subsurfaces $Y_i\subset\partial M$ such that $d_{Y_i}(m_i,\mu)\longrightarrow\infty$ and $d\subset\partial Y_i$. Consider a simple closed curve $c\subset V_d\cap\partial M$ intersecting $d$. Since $\{\rho_{i|\pi_1(V_d)}\}$ converges, $\{\len_{\rho_i}(c)\}$ is bounded.  Then we have $d_{Y_i}(m_i,c)\longrightarrow\infty$ (for $d_{Y_i}(m_i, \mu) \longrightarrow \infty$) and it follows from \cite[Theorem 2.5]{minsky-gt} that $\len_{\rho_i}(d)\longrightarrow 0$.

Suppose now that $\{m(m_i,d_1,\mu)\}$ is bounded.
Since $\{m(\sigma(m_i),d_1,\mu)\}\longrightarrow\infty$ by assumption, $\len_{\rho_i}(d)\longrightarrow 0$ by \cite[Short Curve Theorem]{Mi}. We add to $P$ a thin regular neighbourhood of $d$ on $\partial V_d$ and we can verify as above that the hypotheses of \cref{relative convergence} are satisfied for $(V_d,P)$.
\end{proof}

Now we are in the situation of \cref{long setting}, and we  use its notations. By Lemma \ref{cusp on one side}, $g(F_1)$ lies on the inward side of the cusp corresponding to $\rho_\infty(d)$, and $g(F_j)$ lies on the outward side for every $j=2, \dots , p$. 
Then  \cref{embedded core => bounded} implies that $\{m(\sigma(m_i),d_1,\mu)\}$ is bounded. 
This contradicts our assumption.
\end{proof}

\section{The proof of \cref{main}}	\label{conclusion}
Now we shall complete the proof of \cref{main}.
By \cref{covering,strongly untwisted}, we can assume that every $M$ is strongly untwisted.
Let $L$ be the number provided by \cref{non-extendible}, and consider a sequence $\{m_i\}$ such that $\{m_i^{L+1}=(\iota_* \circ \sigma)^{L+1} m_i\}$ has no convergent subsequence.
By \cref{no parabolics=bounded}, passing to a subsequence, there is a simple closed curve $d_{L+1}\subset\partial M$ such that $m(m_i^{L+1},d_{L+1},\mu)\longrightarrow\infty$. Then we have $m(\sigma(m_i^{L}),\iota(d_{L+1}),\mu)\longrightarrow\infty$.
By \cref{unbounded has a root}, passing to a further subsequence, there is an incompressible annulus $A_{L}$ bounded by $\iota(d_{L+1})$ and another simple closed curve $d_{L}\subset\partial M$ with $m(m_i^L,d_L,\mu)\longrightarrow\infty$.
Repeating this, we get a family of simple closed curves $\{ d_k, 0\leq k\leq L+1\}$ such that $d_{k}\cup \iota(d_{k+1})$ bounds an incompressible annulus.
This means that an annular neighbourhood of $\iota(d_{L+1})$ is $L$-time vertically extendible, contradicting \cref{non-extendible}.
This completes the proof of \cref{main}.

\bibliographystyle{acm}
 \bibliography{bounded}

\end{document}